\documentclass[10pt,a4paper,american, reqno]{amsart}
\usepackage{graphicx} % Required for inserting images
\usepackage{amsmath}
\usepackage{amssymb}
\usepackage{hyperref}
\hypersetup{ colorlinks = true, urlcolor = blue, linkcolor = blue, citecolor = red }
\usepackage{xcolor}
\usepackage{yhmath}

\usepackage{amstext}
\usepackage{amsbsy}
\usepackage{amsfonts,amsthm,latexsym,marvosym,xfrac}
\usepackage{etoolbox}

\theoremstyle{plain}
\newtheorem{theorem}{Theorem}[section]
\newtheorem{lemma}[theorem]{Lemma}
\newtheorem{corollary}[theorem]{Corollary}

\newtheorem{definition}[theorem]{Definition}

\theoremstyle{remark}
\newtheorem{remark}[theorem]{Remark}

\newcommand{\na}{\nabla}
\newcommand{\Om}{\Omega}
\newcommand{\C}{\mathbb{C}}
\newcommand{\N}{\mathbb{N}}
\newcommand{\R}{\mathbb{R}}
\newcommand{\pa}{\partial}

\newcommand{\rr}{^R}
\newcommand{\ii}{^I}
\newcommand{\norm}[1]{\lVert#1\rVert}

\DeclareMathOperator{\dv}{div}
\DeclareMathOperator{\re}{Re}
\DeclareMathOperator{\im}{Im}

\DeclareFontFamily{U}{mathx}{}
\DeclareFontShape{U}{mathx}{m}{n}{ <-> mathx10 }{}
\DeclareSymbolFont{mathx}{U}{mathx}{m}{n}
\DeclareFontSubstitution{U}{mathx}{m}{n}

\DeclareMathAccent{\widecheck}{0}{mathx}{"71}

\def\Xint#1{\mathchoice
    {\XXint\displaystyle\textstyle{#1}}%
    {\XXint\textstyle\scriptstyle{#1}}%
    {\XXint\scriptstyle\scriptscriptstyle{#1}}%
    {\XXint\scriptscriptstyle\scriptscriptstyle{#1}}%
    \!\int}
\def\XXint#1#2#3{\setbox0=\hbox{$#1{#2#3}{\int}$}
    \vcenter{\hbox{$#2#3$}}\kern-0.5\wd0}

\def\dashint{\Xint{\raise4pt\hbox to7pt{\hrulefill}}}

\def\XXiint#1#2#3{\setbox0=\hbox{$#1{#2#3}{\iint}$}
    \vcenter{\hbox{$#2#3$}}\kern-0.5\wd0}

\def\Xint#1{\mathchoice
	{\XXint\displaystyle\textstyle{#1}}%
	{\XXint\textstyle\scriptstyle{#1}}%
	{\XXint\scriptstyle\scriptscriptstyle{#1}}%
	{\XXint\scriptstyle\scriptscriptstyle{#1}}%
	\!\int}
\def\XXint#1#2#3{{\setbox0=\hbox{$#1{#2#3}{\int}$}
		\vcenter{\hbox{$#2#3$}}\kern-.5\wd0}}

\def\YYint#1#2#3{{\setbox0=\hbox{$#1{#2#3}{\iint}$}
		\vcenter{\hbox{$#2#3$}}\kern-.51\wd0}}
 
\def\Xint#1{\mathchoice
	{\XXint\displaystyle\textstyle{#1}}%
	{\XXint\textstyle\scriptstyle{#1}}%
	{\XXint\scriptstyle\scriptscriptstyle{#1}}%
	{\XXint\scriptscriptstyle\scriptscriptstyle{#1}}%
	\!\int}
\def\XXint#1#2#3{{\setbox0=\hbox{$#1{#2#3}{\int}$ }
		\vcenter{\hbox{$#2#3$ }}\kern-.6\wd0}}

\def\dashint{\Xint-}
\newcommand{\fint}{\dashint}

\title[$p$-Laplace with complex coefficients]{Existence, uniqueness and regularity for elliptic $p$-Laplace systems with complex coefficients}


\author[W. Kim]{Wontae Kim}
\address{Wontae Kim,
Department of Mathematics, Uppsala University,
P.~O.~Box 480, 751 06, Uppsala, Sweden}
\email{wontae.kim@math.uu.se}

\author[M. Vestberg]{Matias Vestberg}
\address{Matias Vestberg,
Department of Mathematics, Uppsala University,
P.~O.~Box 480, 751 06, Uppsala, Sweden}
\email{matias.vestberg@math.uu.se}

 \makeatletter
\patchcmd{\@setaddresses}{\indent}{\noindent}{}{}
\patchcmd{\@setaddresses}{\indent}{\noindent}{}{}
\patchcmd{\@setaddresses}{\indent}{\noindent}{}{}
\patchcmd{\@setaddresses}{\indent}{\noindent}{}{}
\makeatother

\begin{document}

\maketitle
\begin{abstract}
	This paper concerns elliptic systems of $p$-Laplace type with complex valued coefficient and source term. We extend the real valued theory of the elliptic $p$-Laplace equation to the complex valued case.
    We establish the existence and uniqueness of solutions to the Dirichlet problem and prove the Schauder estimate in the case of H\"older continuous coefficients and source terms. We also consider families of coefficient functions parametrized by a complex variable and prove a differentiability result for the map taking the complex parameter to the corresponding solution.
\end{abstract}

\section{Introduction}
In this paper, we study elliptic $p$-Laplace systems in the form
\begin{align}\label{eq:intro}
-\dv (a (\varepsilon^2+|\na u|^2)^\frac{p-2}{2}\na u)=-\dv F\quad\text{in}\quad\Omega,
\end{align}
with coefficient function $a: \Omega \to \C$, solution $u:\Omega \to \C^N$, and source term $F:\Omega \to \C^{N\times n}$ for some $N\geq 1$. Here $p\in(1,\infty)$, $\varepsilon\in [0,1]$ and $\Omega$ is an open bounded subset of $\mathbb{R}^n$ for $n\ge2$. We remark that all results except the differentiability property proved in the last section hold also in the limit case $\varepsilon = 0$. 
As we are working with complex-valued solutions, the inner product used in the weak formulation in the real valued context is replaced by the complex inner product $\langle \cdot,\cdot \rangle$, see Definition~\ref{def_complex_inner} for more details.

A natural question that arises is under what conditions a weak solution exists. For the source term, we can use the same integrability assumption as in the real valued case, namely we assume that $F$ belongs to $L^{p'}(\Om; \C^{N\times n})$. Finding a reasonable assumption on the coefficient function, however, requires some more care. We decompose the coefficient $a$ into its real and imaginary parts as follows:
\[
a(x)=a^R(x)+ia^I(x),
\]
where $a^R(x)$ and $a^I(x)$ are real valued measurable functions defined on $\Omega$. To extend the ellipticity condition naturally, we assume that there exist constants $0<\nu<L<\infty$ such that
\[
\nu\le a^R(x)-|a^I(x)|,\qquad a^R(x)+|a^I(x)|<L \quad\text{for a.e.}\quad x\in \Om.
\]
Note that the imaginary part $a^I$ can take negative values.
The operator
\[ 
 a(x)(\varepsilon^2+|\xi|^2)^\frac{p-2}{2}\xi, \quad \xi\in \C^{N\times n},
\]
corresponding to our system of equations will then satisfy the following structure bound:
\[
\big|a(x)(\varepsilon^2+|\xi|^2)^\frac{p-2}{2} \langle\xi,\xi \rangle\big| = |a(x)|(\varepsilon^2+|\xi|^2)^\frac{p-2}{2} |\xi|^2 \le 2^pL(\varepsilon^p+|\xi|^p).
\]
Meanwhile, the coercivity appears in the real part of the operator:
\[
\re \Big( a(x)(\varepsilon^2+|\xi|^2)^\frac{p-2}{2} \langle\xi,\xi \rangle\Big) = a_R(x)(\varepsilon^2+|\xi|^2)^\frac{p-2}{2}|\xi|^2\ge \nu( \varepsilon^2 +|\xi|^2)^\frac{p-2}{2}|\xi|^2.
\]
The monotonicity of the operator turns out to depend strongly on the relation between the real and imaginary parts of $a$. With the additional condition
\[
c_1a^R(x)-c_2|a^I(x)|>\nu
\]
where $c_1,c_2$ appear in Lemma~\ref{str1} and Lemma~\ref{str2}, the real part of the quantity
\[
\langle a(x)(\varepsilon^2+|\xi_1|^2)^\frac{p-2}{2}\xi_1-a(x)(\varepsilon^2+|\xi_2|^2)^\frac{p-2}{2}\xi_2  ,\xi_1-\xi_2\rangle,
\]
is bounded from below by $|\xi_1-\xi_2|^p$ multiplied with a constant, see the proof in Theorem~\ref{uniqueness} and Theorem~\ref{thm_ex} for details. Under these conditions, the existence of a weak solution to the Dirichlet boundary value problem can be established by extending the real valued monotone operator theory to the complex valued setting, see Theorem~\ref{thm_ex}. Moreover, the uniqueness result in Theorem~\ref{uniqueness} follows directly from an energy estimate. To the best of our knowledge, the existence theory for the complex valued elliptic $p$-Laplace equation is new.

Next, we investigate the regularity properties of weak solutions. First, consider a weak solution $v\in W^{1,p}(\Omega; \C^{N})$ to the model case
\begin{align}\label{eq:model}
-\dv\big( (\varepsilon^2+|\na v|^2)^\frac{p-2}{2}\na v \big)= 0,
\end{align}
in $\Omega$. Decomposing $v$ into real and imaginary parts as
\[
v=v^R+iv^I,
\]
where $v^R,v^I\in W^{1,p}(\Omega;\mathbb{R}^N)$, we define
\[
w := (v^R, v^I)\in W^{1,p}(\Omega;\mathbb{R}^{2N}).
\]
Then, as we show in Section \ref{sec:holdcont}, $w$ is a real valued weak solution to
\begin{align*}\label{eq:real_val}
-\dv \big((\varepsilon^2+|\na w|_{\mathbb{R}^{2N\times n}}^2)^\frac{p-2}{2}\na w\big) = 0,
\end{align*}
in $\Omega$. 
Hence, by \cite{uhlenbeck1977regularity}, $\na w$ is locally H\"older continuous, and thus so is $\na v$. By comparing solutions of \eqref{eq:intro} to solutions of the model case \eqref{eq:model} with coinciding boundary values on suitably scaled balls, we are able to obtain the classical Schauder estimate in the case when $a$ and $F$ are H\"older continuous.

The existence and uniqueness results for the linear equation ($p=2$) with complex coefficient were established in \cite{cohen2015approximation,cohen2011analytic}. The existence result therein is based on the Lax-Milgram theorem, while our result extends them to the range $p\in (1,\infty)$, and is based on the monotone operator theory, see \cite{lions1969quelques, MR1422252}. For the Schauder estimate in the real valued case, we refer to \cite{MR1230384,coscia1999holder,misawa2002local,Ya1989,acerbi2002regularity,kuusi2012new,de2023nonuniformly,kuusi2013gradient,bögelein2023holdercontinuitygradientsolutions,bogelein2012holder,MR1962933}.

Apart from the above mentioned topics, we have also investigated how the solution $u$ depends on the choice of the coefficient function $a$. More specifically, we consider the setting where $a$ is parametrized by a complex variable, i.e. $ a: \C \times \Omega \to \C$. The linear case $p=2$ with affine dependence on $z$ was investigated in \cite{cohen2011analytic} where the authors show that the map taking $z$ to the unique solution $u(z)$ to the Dirichlet problem with coefficient function $a(z,\cdot)$ is complex differentiable as a map from $\C$ into $H^1_0(\Omega;\C)$. 

In the case $p\neq 2$ it turns out that such a strong differentiability property is unlikely to hold except perhaps in some rare special cases. Nevertheless, for $\varepsilon > 0$ and rectangular domains, we are able to show that $z \mapsto u(z)$ is differentiable as a map into $H^1_0(\Omega;\C)$ along line segments in the complex plane, under the assumptions that the coefficients $a(z,\cdot)$ form a uniformly bounded set in a H\"older space, $a^I$ is sufficiently small compared to $a^R$, and $z\mapsto a(z,\cdot)$ is complex differentiable as a map into $L^\infty(\Omega;\C)$. For the precise result, we refer to Theorem \ref{diff_thm}. Note that this is somewhat stronger than simply having directional derivatives. 

To prove the differentiability we first establish the boundedness of the difference quotients $[u(z+h)-u(z)]/h$ in a Sobolev norm. This allows us to utilize weak compactness to extract a convergent subsequence for any sequence $h_j\to 0$. If $h_j$ is confined to a line segment it turns out that the limit satisfies a condition reminiscent of a weak formulation of a PDE. This condition uniquely determines the weak limit and a further calculation shows that the limit function is also the derivative of $z\mapsto u(z)$ along the specific line segment.

The original motivation in \cite{cohen2011analytic} to study the question of complex differentiability in the linear case was to prove that one can jointly approximate the solutions $u(z)$ by a certain series expansion, even in the case where $a$ and hence $u$ depends on countably many complex parameters. This result has later been applied to obtain error estimates for DeepONets, see \cite{LanthMishKarn}.

\vspace{2mm}
\noindent{\bf Acknowledgments.} This work was supported by the Wallenberg AI, Autonomous Systems and Software Program (WASP) funded by the Knut and Alice Wallenberg Foundation.

\section{Preliminary}

Let $i \in \C$ be the standard imaginary unit and let $n\ge1$ be a natural number. 

\begin{definition}
      For any positive integer $N\ge1$, we denote $F\in \mathbb{R}^{N\times n}$ if $F=(F_1,...,F_N)$ for some $F_j\in \mathbb{R}^n$, $1\le j\le N$. For any $c\in \mathbb{R}$, we denote $cF=(cF_1,...,cF_2)$.
     For $F,G\in \mathbb{R}^{N\times n}$, we generalize the inner product in $\mathbb{R}^{N\times n}$ by
     \[
     F\cdot G= \sum_{1\le j\le N}F_j\cdot G_j,
     \]
     where the notion $\cdot$ on the right hand side is the usual inner product in Euclidean space. The norm  in $\mathbb{R}^{N\times n}$ is defined as
     \[
     |F|_{\mathbb{R}^{N\times n}}=\Biggl(\sum_{1\le j\le N}|F_j|^2\Biggr)^\frac{1}{2},
     \]
     where $|F_j|$ is the usual norm in Euclidean space.
\end{definition}

\begin{definition}\label{def_complex_inner}
    We denote $F\in \C^{N\times n}$ if $F=F\rr+iF\ii$ holds for some $F\rr,F\ii\in \mathbb{R}^{N\times n}$. For any $c\in \mathbb{R}$, we denote $cF=cF\rr+icF\ii$.
    We naturally extend the complex inner product to elements $F,G\in \C^{N\times n}$ as
    \[
    \langle F,G \rangle 
    =\sum_{1\le j\le N} (F_j\rr\cdot G_j\rr+F_j\ii\cdot G_j\ii+i(-F_j\rr\cdot G_j\ii+ F_j\ii\cdot G_j\rr )  ),
    \]
    where the notation $\cdot$ on the right hand side is the inner product in Euclidean space.
    The norm  of $F$, denoted by $|F|$, is defined as
     \[
     |F|_{\C^{N\times n}}=(|F\rr|_{\mathbb{R}^{N\times n}}^2+|F\ii|_{\mathbb{R}^{N\times n}}^2)^\frac{1}{2}.
     \]
\end{definition}

\begin{definition}
    For $F\in \C^{N\times n}$, we denote
    \[
    \widehat{F}=(F_1\rr,...,F_N\rr, F_1\ii,...,F_N\ii),\quad \widecheck{F}=( -F_1\ii,...,-F_N\ii,F_1\rr,...,F_N\rr).
    \]
    We remark that $\widehat{F},\widecheck{F}\in \mathbb{R}^{2N\times n}$ and 
    \[
    \langle F,G \rangle = \widehat{F}\cdot \widehat{G}+i \widehat{F}\cdot \widecheck{G},
    \]
    where $\cdot$ on the right hand side is the inner product in $\mathbb{R}^{2N\times n}$. 
\end{definition}
Note that
    \begin{align}\label{equiv_norms}
    |\widehat{F}|_{\mathbb{R}^{2N\times n}} = |\widecheck{F}|_{\mathbb{R}^{2N\times n}}=|F|_{\C^{N\times n}}.
    \end{align}
In expressions such as those above, we may omit the subindexed vector space when it is clear from the context.

\begin{definition}
    For $F\in \C^{1\times 1}=\C$ with $F=F\rr+iF\ii$ for $F\rr,F\ii\in \mathbb{R}$, we define the projection to the real and imaginary axis by
    \[
    \re (F)=F\rr,\qquad \im(F)=F\ii.
    \]
\end{definition}

\begin{definition}
    Let $\Om$ be an open subset in $\mathbb{R}^n$ with $n\ge2$. For $p\in[1,\infty)$, a measurable map $F:\Om\mapsto \mathbb{R}^{N\times n}$ belongs to $L^p(\Om;\mathbb{R}^{N\times n})$ if 
    \[
    \int_{\Om}|F|^p\,dx<\infty.
    \]
\end{definition}

\begin{definition}
     Let $\Om$ be an open subset in $\mathbb{R}^n$ with $n\ge2$.
     For $p\in[1,\infty)$, $u\in W_0^{1,p}(\Om;\mathbb{C}^N)$ if
    \[
    \int_\Om |u|^p+|\na u|^p\,dx<\infty,
    \]
    where the first norm $|u|$ is the norm in $\mathbb{C}^N$ while the second norm $|\na u|$ is the norm in $\C^{N\times n}$. Here, we have set $u=u\rr+iu\ii$ with $u\rr,u\ii\in W_0^{1,p}(\Om;\mathbb{R}^N)$ and $\na u= \na u\rr+i\na u\ii$ with $\na u\rr=(\na u\rr_1,...,\na u\rr_N)$, $\na u\ii=(\na u\ii_1,...,\na u\ii_N)$, $\na u\rr_j=(\pa_1 u\rr_j,..., \pa_nu\rr_j)$ and $\na u\ii_j=(\pa_1 u\ii_j,..., \pa_nu\ii_j)$.
\end{definition}

\begin{definition}
    For $u\in W_0^{1,p}(\Om;\C^N)$, we will denote
    \begin{align*}
            \widehat{u}&=(u\rr_1,...,u\rr_N,u\ii_1,...,u\ii_N), &\na\widehat{u}&= (\na u\rr_1,...,\na u\rr_N,\na u\ii_1,...,\na u\ii_N),\\
             \widecheck{u}&=(-u\ii_1,...,-u\ii_N,u\rr_1,...,u\rr_N),   & \na\widecheck{u}&= (-\na u\ii_1,...,-\na u\ii_N,\na u\rr_1,...,\na u\rr_N).
    \end{align*}
\end{definition}

\section{Complex valued $p$-Laplace system}
 Let $n\ge2$ and let $\Om$ be an open bounded subset in $\mathbb{R}^n$.
Consider the function 
\begin{align}\label{coeff}
    a(x)=a\rr(x)+ia\ii(x)
\end{align}
with $a\rr,a\ii:\Om\mapsto\mathbb{R}$ be bounded measurable functions such that there exist $\nu,L>0$ satisfying
\begin{align}\label{ell}
    \nu<a\rr(x)-|a\ii(x)|,\quad a\rr(x)+|a\ii(x)|<L\quad\text{for a.e.}\quad x\in \Om.
\end{align}
Note that the above condition reduces to the usual ellipticity condition if the coefficient is real valued.
For complex valued coefficients, we define solutions to the  elliptic $p$-Laplace system as follows.
\begin{definition}
    Suppose $p\in(1,\infty)$, $\varepsilon\in[0,1]$, and let $a(x)$ satisfy \eqref{coeff} and \eqref{ell}. Suppose that $F\in L^{p'}(\Om;\mathbb{C}^{N\times n})$ where $p'$ is the H\"older conjugate of $p$. Then $u\in W^{1,p}(\Om;\mathbb{C}^N)$ is a weak solution to
    \begin{align}\label{eq}
        -\dv(a(x)( \varepsilon^2 + |\na u|^2)^\frac{p-2}{2}\na u)=-\dv F  \quad\text{in}\quad\Om
    \end{align}
    if for every $\varphi\in W_0^{1,p}(\Om;\C^N)$, there holds
    \[
    \int_{\Om} a(x)( \varepsilon^2 + |\na u|^2)^\frac{p-2}{2}\langle \na u,\na \varphi\rangle\,dx= \int_\Om \langle F,\na \varphi\rangle \,dx.
    \]
\end{definition}

\section{Structure lemmas}
Before we can prove any results for the complex valued $p$-Laplace system, we establish that the assumptions on the complex coefficient $a(x)$ gives rise to some structure conditions. The following Lemma easily follows from the definition of the complex inner product and the equivalence of norms observed earlier in \eqref{equiv_norms}.
\begin{lemma}\label{str}
    For any $F,G\in\C^{N\times n}$, there holds
    \begin{align*}
        \begin{split}
            &\langle (\varepsilon^2+|F|^2)^\frac{p-2}{2}F-(\varepsilon^2+|G|^2)^\frac{p-2}{2}G, F-G\rangle \\
            &= ( (\varepsilon^2+|\widehat{F}|^2)^\frac{p-2}{2}\widehat{F} - (\varepsilon^2 + |\widehat{G}|^2)^\frac{p-2}{2}\widehat{G} ) \cdot  ( \widehat{F}-\widehat{G}  )\\
            &\qquad+ i((\varepsilon^2+|\widehat{F}|^2)^\frac{p-2}{2}\widehat{F}-(\varepsilon^2+|\widehat{G}|^2)^\frac{p-2}{2}\widehat{G} )\cdot (\widecheck{F}-\widecheck{G}).
        \end{split}
    \end{align*}
\end{lemma}
The real part of the righ-hand side in the previous equation satisfies the following coercivity condition. The proof follows directly from the structure condition of the real valued $p$-Laplace operator.
\begin{lemma}\label{str1}
    For any $F,G\in\C^{N\times n}$, there exists a constant $c_1>0$ such that
    \[
    c_1(\varepsilon^2+|F|^2+|G|^2)^\frac{p-2}{2}|F-G|^2\le ((\varepsilon^2+|\widehat{F}|^2)^\frac{p-2}{2}\widehat{F}-(\varepsilon^2+|\widehat{G}|^2)^\frac{p-2}{2}\widehat{G} ) \cdot  ( \widehat{F}-\widehat{G}  ) 
    \]
\end{lemma}

The following two lemmas also follow from the structure condition of the real valued $p$-Laplace operator and the norm equivalence.
\begin{lemma}\label{str2}
    For any $F,G\in\C^{N\times n}$, there exists a constant $c_2>0$ such that
    \[
    |(\varepsilon^2+|F|^2)^\frac{p-2}{2}F-(\varepsilon^2+|G|^2)^\frac{p-2}{2}G|\le c_2 (\varepsilon^2+|F|^2+|G|^2)^\frac{p-2}{2}|F-G|.
    \]
\end{lemma}

\begin{lemma}\label{str3}
    For any $F,G\in\C^{N\times n}$, there exists $c=c(p)$ such that 
\begin{align*}
    \begin{split}
            c^{-1}|(\varepsilon^2+|F|^2)^\frac{p-2}{4}F-(\varepsilon^2+|G|^2)^\frac{p-2}{4}G|^2
    &\le(\varepsilon^2+|F|^2+|G|^2)^\frac{p-2}{2}|F-G|^2\\
    &\le c|(\varepsilon^2+|F|^2)^\frac{p-2}{4}F-(\varepsilon^2+|G|^2)^\frac{p-2}{4}G|^2.
    \end{split}
\end{align*}
\end{lemma}

\section{Existence}\label{sec:existence}
The aim of this section is to prove the existence and uniqueness of weak solutions to the Dirichlet problem associated with the elliptic system. Besides the ellipticity condition \eqref{ell}, we will also assume that
\begin{align}\label{ell2}
    c_1a\rr(x)-c_2|a\ii(x)|>\nu,
\end{align}
where $c_1,c_2$ are in Lemma~\ref{str1} and Lemma~\ref{str2}, respectively. This condition imposes the room for oscillation of $a\ii$. Note that $c_1\equiv c_2\equiv1$ if $p=2$, which is the first condition of \eqref{ell}. 
As far as we are concerned, the best constants for Lemma~\ref{str1} and Lemma~\ref{str2} remain unknown unless $p=2$. For existing results, we refer to \cite[Lemma 2.1 and Lemma 2.2]{acerbi1989regularity} and \cite[Lemma 8.3]{MR1962933}, which provide $c_1=\tfrac{1}{3}(\tfrac
{1}{3\sqrt{2}})^{p-2}$ for $p\ge2$, and $c_1=p-1$ for $p<2$, and $c_2=p-1$ when $p\ge2$, while $c_2=8$ when $p<2$.

The role of \eqref{ell2} is more visible in the standard energy estimate. We first prove the uniqueness of weak solutions.

\begin{definition}
    For the coefficient function $a:\Om\mapsto \C$, we denote $a\in\mathcal{E}(\Om)$ if it satisfies \eqref{coeff}, \eqref{ell} and \eqref{ell2}, 
\end{definition}

\begin{theorem}\label{uniqueness}
   Let $a\in\mathcal{E}(\Om)$, $F\in L^{p'}(\Om;\C^{N\times n})$ and $g\in W^{1,p}(\Om;\C^N)$ be given. Then there is at most one weak  solution to \eqref{eq} in $g+W_0^{1,p}(\Om;\C^N)$. Here, $u\in g+W_0^{1,p}(\Om;\C^N)$ means $u-g\in W_0^{1,p}(\Om;\C^N)$.
\end{theorem}
\begin{proof}
 Suppose that $u,w\in g+W^{1,p}_0(\Om;\C^N)$ are both weak solutions to \eqref{eq}. Then we have
 \[
 -\dv(a(x)  (  (\varepsilon^2+|\na u|^2)^\frac{p-2}{2}\na u   -  (\varepsilon^2+|\na w|^2)^\frac{p-2}{2}\na w)  )=0.
 \]
Taking $u-w\in W_0^{1,p}(\Om;\C^N)$ as a test function, we have
\[
\int_{\Om} a(x) \langle  (\varepsilon^2+|\na u|^2)^\frac{p-2}{2}\na u   -  (\varepsilon^2+|\na w|^2)^\frac{p-2}{2}\na w), \na u-\na w   \rangle\,dx=0.
\]
Taking the real part of this equation we have by Lemma~\ref{str} that
\begin{align*}
    \begin{split}
        0&=\int_{\Om} a\rr(x) ( (\varepsilon^2+|\na \widehat{u}|^2)^\frac{p-2}{2}\na \widehat{u}   -  (\varepsilon^2+|\na \widehat{w}|^2)^\frac{p-2}{2}\na \widehat{w} \cdot( \na \widehat{u}-\na \widehat{w} )  \,dx\\
        &\qquad-\int_{\Om} a\ii(x) ( (\varepsilon^2+|\na \widehat{u}|^2)^\frac{p-2}{2}\na \widehat{u}   -  (\varepsilon^2+|\na \widehat{w}|^2)^\frac{p-2}{2}\na \widehat{w} \cdot( \na \widecheck{u}-\na \widecheck{w} ) \,dx.
    \end{split}
\end{align*}
Employing Lemma~\ref{str1}, Lemma~\ref{str2} and \eqref{ell2}, we get
\begin{align*}
    \begin{split}
        0&\ge\int_{\Om} c_1a\rr(x) ( \varepsilon^2+|\na u|^2+|\na w|^2  )^\frac{p-2}{2}|\na u-\na w|^2\,dx\\
        &\qquad-\int_{\Om} c_2|a\ii(x)| ( (\varepsilon^2+|\na u|^2 +|\na w|^2 )^\frac{p-2}{2} |\na u-\na w|^2 \,dx\\
        &\ge\int_{\Om} \nu ( \varepsilon^2+|\na u|^2+|\na w|^2  )^\frac{p-2}{2}|\na u-\na w|^2\,dx.
    \end{split}
\end{align*}
 Hence, it follows $\na u\equiv \na w$ and thus, $u\equiv w$.
\end{proof}

Next, we state the existence result in the case of vanishing boundary data.
\begin{theorem}\label{thm_ex}
    Let $a\in\mathcal{E}(\Om)$ and $F\in L^{p'}(\Om;\C^{N\times n})$ be given. Then, there exists a weak solution $u\in W_0^{1,p}(\Om;\C^N)$ to \eqref{eq}.
\end{theorem}

The proof is an adaptation of the standard argument of the real valued monotone operator theory. We first recall the Schauder fixed-point theorem.
\begin{lemma}\label{Brouwer}
    Let $K$ be a nonempty convex compact subset of a Banach space. Then, every continuous function from $K$ to itself has a fixed point.
\end{lemma}

Next, we state and prove a variant a standard corollary of the fixed point theorem, adapted to the complex setting. The strategy of the proof is similar to the real valued case.
\begin{lemma}\label{Brouwer_cor}
    Let $m\ge1$ be a positive integer and $g:\C^{m}\to \C^{m}$  be a continuous map. Suppose there exists $\rho>0$ such that
    \begin{align*}
        \re \langle g(z), z \rangle\ge 0\quad\text{for all}\quad z\in\C^{m}\quad\text{with}\quad |z|=\rho.
    \end{align*}
    Then there exists $z_0\in \C^{m}$ with $|z_0|\le \rho$ such that $g(z_0)=0.$
\end{lemma}

\begin{proof}
    We prove this by contradiction. Then, $g(z)\ne 0$ in $K=\{ z\in \C^{m} : |z|\le \rho\}$ and the following function
    \begin{align*}
        h(z)=-\rho\frac{g(z)}{|g(z)|}
    \end{align*}
    is well defined in $K$. By Lemma~\ref{Brouwer}, we find $w\in K$ such that $h(w)=w.$
    Note that since the range of $h$ lies in the sphere $\{ z\in \C^{1\times m} :|z|=\rho \}$, we have $|w|=\rho$. Thus it follows that
    \begin{align*}
       \langle h(w), w  \rangle=|w|^2=\rho^2>0.
    \end{align*}
    On the other hand, by the assumption on $g$ we have
    \begin{align*}
        \re \langle h(w), w  \rangle =\re \left\langle-\rho\frac{g(w)}{|g(w)| }, w\right\rangle=-\frac{\rho}{|g(w)|}\re\langle g(w), w\rangle\le 0.
    \end{align*}
   This is a contradiction and thus the conclusion of this lemma is true.    
\end{proof}

We now prove the existence result.

\begin{proof}[Proof of Theorem~\ref{thm_ex}]
Since $W_0^{1,p}(\Om;\C^N)$ coincides with all the functions of the form $\varphi^R+i\varphi^I$ where $\varphi^R,\varphi^I\in W_0^{1,p}(\Om;\mathbb{R}^N)$, it follows that $W_0^{1,p}(\Om;\C^N)$ is a separable space. Therefore, there exists a linearly independent countable subset $\{\varphi_j\}_{j\in\mathbb{N}}$ in $ W_0^{1,p}(\Om;\C^N)$ whose linear combinations are dense. For fixed $m\in \mathbb{N}$, we let $W_m$ denote the linear combinations of $\{\varphi_j\}_{1\le j\le m}$ over $\C$. For each $z\in \C^{m}$ with components $z_1,...,z_m\in \mathbb{C}$, consider the map defined as
\[
u(z)= \sum_{1\le j\le m}z_j\varphi_j\in W_m.
\]
We can now define the map $g=(g_1,...,g_m):\C^{m} \to \C^{m}$ where each component is given by
\[
g_j(z)=\int_{\Om} (a(\varepsilon^2+|\na u(z)|^2)^\frac{p-2}{2}\langle \na u(z), \na \varphi_j \rangle -\langle F, \na \varphi_j\rangle) \,dx 
\]
It is easy to see that
\[
\langle g(z),z\rangle =\int_{\Om} (a(\varepsilon^2+|\na u(z)|^2)^\frac{p-2}{2}\langle \na u(z), \na u(z) \rangle -\langle F, \na u(z) \rangle )\,dx.
\]
Taking the real part of this equation, we have that
\begin{align*}
    \begin{split}
        \re\langle g(z),z\rangle 
        &=\int_{\Om} a_1(\varepsilon^2+|\na u(z)|^2)^\frac{p-2}{2}|\na u(z)|^{2}\,dx-\re\int_\Om\langle F, \na u(z) \rangle \,dx\\
        &=\int_{\Om} a_1(\varepsilon^2+|\na u(z)|^2)^\frac{p-2}{2}|\na u(z)|^{2}\,dx-\int_\Om \widehat{F} \cdot \na\widehat{u}(z) \,dx\\
        &\ge\int_{\Om} \nu(\varepsilon^2+|\na u(z)|^2)^\frac{p-2}{2}|\na u(z)|^{2}\,dx-\int_\Om |F|| \na u(z)| \,dx,
    \end{split}
\end{align*}
where to obtain the last inequality, we used \eqref{ell}. In order to further estimate this expression we need to use different methods depending on the range of $p$.

\noindent \textit{Case $p\ge2$.} In this case, we have from H\"older inequality that
\begin{align*}
    \begin{split}
        \re\langle g(z),z\rangle 
        &\ge \nu\int_{\Om} |\na u(z)|^p\,dx-\left(\int_{\Om}| F|^{p'} \,dx\right)^\frac{p-1}{p}\left(\int_{\Om}| \na u(z)|^{p} \,dx\right)^\frac{1}{p}\\
        &\ge \left(\nu\left(\int_{\Om} |\na u(z)|^p\,dx\right)^\frac{p-1}{p}-\left(\int_{\Om}| F|^{p'} \,dx\right)^\frac{p-1}{p}\right)\left(\int_{\Om}| \na u(z)|^{p} \,dx\right)^\frac{1}{p}.
    \end{split}
\end{align*}
From the last estimate we can conclude that for some sufficiently large $\rho$,
\begin{align}\label{re-inner-prod-est}
\re\langle g(z),z\rangle \ge 0 \quad\text{for all}\quad |z|\ge\rho.
\end{align}
Indeed, we have
\[
0<\min_{|z|=1} \int_\Om |\na u(z)|^p\,dx,
\]
due to the linear independence of $\{\varphi_j\}_{1\le j\le m}$. This observation together with a scaling argument confirms \eqref{re-inner-prod-est}.

\noindent \textit{Case $p<2$.} By H\"older's inequality we have
\begin{align*}
    \begin{split}
        &\int_{\Om} |\na u(z)|^p\,dx\\
        &=\int_{\Om} (\varepsilon^2+|\na u(z)|^2)^\frac{(p-2)p}{4}|\na u(z)|^p(\varepsilon^2+|\na u(z)|^2)^\frac{(2-p)p}{4}\,dx\\
        &\le \left(\int_{\Om} (\varepsilon^2+|\na u(z)|^2)^\frac{p-2}{2}|\na u(z)|^2\,dx\right)^\frac{p}{2}\left(\int_{\Om} (\varepsilon^2+|\na u(z)|^2)^\frac{p}{2}\,dx\right)^\frac{2-p}{2}\\
        &\le\left(\int_{\Om} (\varepsilon^2+|\na u(z)|^2)^\frac{p-2}{2}|\na u(z)|^2\,dx\right)^\frac{p}{2}\left(2^\frac{p}{2}\int_{\Om} (\varepsilon^p+|\na u(z)|^p)\,dx\right)^\frac{2-p}{2}.
    \end{split}
\end{align*}
Therefore for sufficiently large $|z|$, we get
\[
\int_{\Om} |\na u(z)|^p\,dx\le \left(\int_{\Om} (\varepsilon^2+|\na u(z)|^2)^\frac{p-2}{2}|\na u(z)|^2\,dx\right)^\frac{p}{2}\left(2^{\frac{p}{2}+1}\int_{\Om} |\na u(z)|^p\,dx\right)^\frac{2-p}{2},
\]
or equivalently,
\[
\int_{\Om} |\na u(z)|^p\,dx\le 4\int_{\Om} (\varepsilon^2+|\na u(z)|^2)^\frac{p-2}{2}|\na u(z)|^2\,dx.
\]
Therefore, we again end up with the estimate
\begin{align*}
    \begin{split}
        \re\langle g(z),z\rangle 
        &\ge\left(c\left(\int_{\Om} |\na u(z)|^p\,dx\right)^\frac{p-1}{p}-\left(\int_{\Om}| F
        |^{p'} \,dx\right)^\frac{p-1}{p}\right)\left(\int_{\Om}| \na u(z)|^{p} \,dx\right)^\frac{1}{p}\\
        &\ge0,
    \end{split}
\end{align*}
for $|z| \geq \rho$ with sufficiently large $\rho>0$.

For both cases, applying Lemma~\ref{Brouwer_cor}, we find $u_m\in W_m$ such that
\begin{align}\label{fin_eq}
    \int_{\Om} a(\varepsilon^2+|\na u_m|^2)^\frac{p-2}{2}\langle \na u_m, \na \varphi_j \rangle -\langle F, \na \varphi_j\rangle \,dx =0\quad\text{for all}\quad 1\le j\le m.
\end{align}
As the above inequality implies
\[
\int_{\Om} a(\varepsilon^2+|\na u_m|^2)^\frac{p-2}{2}|\na u_m|^{2} \,dx =\int_\Om \langle F, \na u_m\rangle\,dx,
\]
we take the real part of both sides to get
\begin{align*}
    \int_{\Om} a_1(\varepsilon^2+|\na u_m|^2)^\frac{p-2}{2}|\na u_m|^{2} \,dx 
    &=\re\int_\Om \langle F, \na u_m\rangle\,dx\\
    &\le \int_\Om | F| |\na u_m|\,dx\\
    &\le \left(\int_\Om | F|^{p'}\,dx\right)^\frac{1}{p'} \left( \int_\Om|\na u_m|^p\,dx\right)^\frac{1}{p}.
\end{align*}
If $p\ge2$, the last estimate and \eqref{ell}, imply that
\[
    \int_{\Om} |\na u_m|^{p} \,dx \le c\int_\Om | F|^{p'}\,dx,
\]
for all $m\in\mathbb{N}$ where $c=c(p,\nu)$. 
On the other hand, if $p< 2$, to have a uniform upper bound for the norm of the gradients, it is sufficient to consider the case when
\[
\int_{\Om}\varepsilon^p\,dx\le \int_{\Om}|\na u_m|^p\,dx.
\]
Then, we have
\begin{align*}
    \begin{split}
        \int_{\Om} |\na u_m|^p\,dx
        &\le\left(\int_{\Om} (\varepsilon^2+|\na u_m|^2)^\frac{p-2}{2}|\na u_m|^2\,dx\right)^\frac{p}{2}\left(2^\frac{p}{2}\int_{\Om} (\varepsilon^p+|\na u_m|^p)\,dx\right)^\frac{2-p}{2}\\
        &\le\left(\int_{\Om} (\varepsilon^2+|\na u_m|^2)^\frac{p-2}{2}|\na u_m|^2\,dx\right)^\frac{p}{2}\left(2^{\frac{p}{2}+1}\int_{\Om} |\na u_m|^p\,dx\right)^\frac{2-p}{2}
    \end{split}
\end{align*}
Therefore, in any case we obtain
\begin{align}\label{fin_est}
    \int_{\Om} |\na u_m|^{p} \,dx \le c\int_\Om (| F|^{p'}+1)\,dx,
\end{align}
for all $m\in\mathbb{N}$ where $c=c(p,\nu)>0$. By the compactness of the embedding $W_0^{1,p}(\Om;\C^N) \rightarrow L^p(\Om;\C^N)$ we see that there is a subsequence, still labelled $(u_m)$, and an element $u \in L^p(\Om;\C^N)$ such that
\begin{align*}
 u_m \to u \quad \text{strongly in} \quad L^{p}(\Om;\C^{N\times n}).
\end{align*}
Since $W_m\subset W_0^{1,p}(\Om;\C^N)$ for all $m$, the uniform bound \eqref{fin_est} implies that also a subsequence of the gradients converges weakly. In fact, we see that $u\in W_0^{1,p}(\Om;\C^N)$ and for a suitable subsequence,
\[
       \na u_m\to \na u\quad\text{weakly in}\quad L^{p}(\Om;\C^{N\times n}).
\]
In particular the subsequence $(u_m)$ converges weakly to $u$ in $W_0^{1,p}(\Om;\C^N)$. On the other hand, as we have
\[
(\varepsilon^2+|\na u_m|^2)^\frac{p-2}{2}|\na u_m|\le(\varepsilon^2+|\na u_m|^2)^\frac{p-1}{2},
\]
it follows from \eqref{fin_est} that by passing to yet another subsequence there is an element $\mathcal{A}$ in $L^{p'}(\Om;\C^{N\times n})$ such that
\[
a(\varepsilon^2+|\na u_m|^2)^\frac{p-2}{2}\na u_m\to \mathcal{A}\quad \text{weakly in}\quad L^{p'}(\Om;\C^{N\times n}).
\]
In order to complete the proof, it only remains to prove that
\begin{align}\label{div-eq}
\nabla \cdot \mathcal{A} = \nabla \cdot (a(\varepsilon^2+|\na u|^2)^\frac{p-2}{2}\na u) \quad\text{in}\quad W^{-1,p'}(\Om;\C^N).
\end{align}
Fix $\psi \in W^{1,p}(\Om;\C^N)$.
Employing Lemma~\ref{str}, Lemma~\ref{str1}, Lemma~\ref{str2} and \eqref{ell2}, it follows that
\begin{align*}
    \begin{split}
        &\re\int_{\Om} a \langle(\varepsilon^2+|\na u_m|^2)^\frac{p-2}{2}\na u_m-(\varepsilon^2+|\na \psi|^2)^\frac{p-2}{2}\na \psi, \na u_m-\na \psi\rangle\,dx\\
        &\ge \int_{\Om}(c_1a_1(x) - c_2|a_2(x)|)(\varepsilon^2+|\na u_m|^2+|\na \psi|^2)^\frac{p-2}{2}|\na u_m-\na \psi|^2\,dx\\
        &\ge 0.
    \end{split}
\end{align*}
By splitting the term on the first row we obtain
\begin{align*}
    \begin{split}
        0
        &\le \re\int_{\Om} a \langle(\varepsilon^2+|\na u_m|^2)^\frac{p-2}{2}\na u_m, \na u_m-\na u\rangle\,dx
        \\
        &\qquad+\re\int_{\Om} a \langle(\varepsilon^2+|\na u_m|^2)^\frac{p-2}{2}\na u_m, \na u-\na \psi \rangle\,dx
        \\
        &\qquad-\re\int_{\Om} a \langle(\varepsilon^2+|\na \psi|^2)^\frac{p-2}{2}\na \psi, \na u_m-\na u\rangle\,dx
        \\
        &\qquad-\re\int_{\Om} a \langle(\varepsilon^2+|\na \psi|^2)^\frac{p-2}{2}\na \psi, \na u-\na \psi \rangle\,dx\\
        &=I+II+III+IV.
    \end{split}
\end{align*}
In order to estimate $I$, we choose elements $v_m$ in $W_m$ such that
\[
v_m\to u\quad\text{strongly in}\quad W_0^{1,p}(\Om;\mathbb{C}^N).
\]
Such $v_m$ exist since $u$ is the weak limit of the sequence $(u_m)$ and since the weak and strong closure of the subspace $\cup_{m\in \N} W_m$ coincide. Then we obtain
\begin{align*}
    \begin{split}
        I
        &=\re\int_{\Om} a (\varepsilon^2+|\na u_m|^2)^\frac{p-2}{2}  \langle\na u_m, \na u_m-\na v_m\rangle\,dx\\
        &\qquad+\re\int_{\Om} a (\varepsilon^2+|\na u_m|^2)^\frac{p-2}{2}  \langle\na u_m, \na v_m-\na u\rangle\,dx.
    \end{split}
\end{align*}
As we have \eqref{fin_eq}, the first term on the right hand side is equivalent to
\[
\re\int_{\Om}\langle F, \na  u_m-\na v_m\rangle\,dx,
\]
which converges to $0$ as $m$ goes to infinity due to the weak convergence of $u_n$ and $v_n$. We apply H\"older's inequality to the second term on the right hand side. Then, we get
\begin{align*}
    \begin{split}
        &\re\int_{\Om} a (\varepsilon^2+|\na u_m|^2)^\frac{p-2}{2} \langle\na u_m, \na v_m-\na u\rangle\,dx\\
        &\le L\int_{\Om}  (\varepsilon^2+|\na u_m|^2)^\frac{p-1}{2}|\na v_m-\na u|\,dx\\
        &\le c \left(\int_\Om |F|^{p'}+1\,dx\right)^\frac{1}{p'}\left(\int_{\Om}|\na v_m-\na u|^p\,dx\right)^\frac{1}{p},
    \end{split}
\end{align*}
where we used \eqref{ell} and \eqref{fin_est}. The last term converges to $0$ since $\na v_m$ converges strongly to $\na u$ as $m$ goes to infinity.  Therefore we conclude
\[
\lim_{m\to\infty}I=0.
\]
By the weak convergence, we have
\[
\lim_{m\to\infty}II=\re\int_{\Om}\langle \mathcal{A},\na u-\na \psi` \rangle\,dx
\]
and
\[
\lim_{m\to\infty}III=0.
\]

Combining these estimates, it follows that
\[
\re\int_{\Om} a (\varepsilon^2+|\na \psi|^2)^\frac{p-2}{2} \langle\na \psi, \na u-\na \psi \rangle\,dx\le \re\int_{\Om}\langle \mathcal{A},\na u-\na \psi \rangle\,dx,
\]
where $\psi \in W^{1,p}(\Om;\C^N)$ is arbitrary. Taking $s>0$ and $\psi = u - s\varphi$ for any $\varphi \in W_0^{1,p}(\Om;\C^N)$, we obtain
\[
\re\int_{\Om} a (\varepsilon^2+|\na u - s\na \varphi|^2)^\frac{p-2}{2} \langle(\na u - s\na \varphi), \na \varphi\rangle\,dx\le \re\int_{\Om}\langle \mathcal{A},\na \varphi\rangle\,dx.
\]
Letting $s$ to $0$, we get
\[
\re\int_{\Om} a (\varepsilon^2+|\na u|^2)^\frac{p-2}{2}  \langle\na u, \na \varphi\rangle\,dx\le \re\int_{\Om}\langle \mathcal{A},\na \varphi\rangle\,dx.
\]
Again, since $\varphi$ is arbitrary, by substituting $-\varphi$, we have
\[
\re\int_{\Om} a (\varepsilon^2+|\na u|^2)^\frac{p-2}{2}  \langle\na u, \na \varphi\rangle\,dx=\re\int_{\Om}\langle \mathcal{A},\na \varphi\rangle\,dx.
\]
On the other hand, by replacing $\varphi$ with $i\varphi$, it follows that
\begin{align*}
    \begin{split}
      \im\int_{\Om} a (\varepsilon^2+|\na u|^2)^\frac{p-2}{2} \langle\na u, \na \varphi\rangle\,dx
      &=\re -i\int_{\Om} a (\varepsilon^2+|\na u|^2)^\frac{p-2}{2}  \langle\na u, \na \varphi\rangle\,dx\\
      &=\re \int_{\Om} a (\varepsilon^2+|\na u|^2)^\frac{p-2}{2}  \langle\na u, i\na \varphi\rangle\,dx\\
      &=\re\int_{\Om}\langle \mathcal{A}, i\na \varphi\rangle\,dx\\
      &=\im\int_{\Om}\langle \mathcal{A}, \na \varphi\rangle\,dx.
    \end{split}
\end{align*}
Hence, we have confirmed \eqref{div-eq}. This completes the proof.
\end{proof}

As a consequence, we also have the existence of solutions to the Dirichlet problem with non-zero boundary values.
\begin{corollary}\label{thm_ex_non_bd}
   Let $a\in\mathcal{E}(\Om)$, $F\in L^{p'}(\Om;\C^{N\times n})$ and $g\in W^{1,p}(\Om;\C^N)$ be given. Then, there exists a weak solution $u$ to \eqref{eq} with $u-g\in W_0^{1,p}(\Om;\C^N)$.
\end{corollary}

\begin{proof}
    It is sufficient to find a weak solution $w\in W_0^{1,p}(\Om;\C^N)$ to 
    \[
    -\dv(a(\varepsilon^2+|\na w+\na g|^2)^\frac{p-2}{2}(\na w+\na g))= -\dv F \quad\text{in}\quad\Om.
    \]
    Then $u=w+g$ is the desired weak solution. We will initially prove some estimates which are valid for all $w\in W_0^{1,p}(\Om;\C^N)$. Note that 
    \begin{align*}
        \begin{split}
            &a( \varepsilon^2+|\na w +\na g|^2 )^\frac{p-2}{2}  \langle (\na w+\na g),\na w  \rangle\\
            &= a( \varepsilon^2+|\na w +\na g|^2 )^\frac{p-2}{2}|\na w+\na g|^2\\
            &\qquad - a( \varepsilon^2+|\na w +\na g|^2 )^\frac{p-2}{2}  \langle (\na w+\na g),\na g  \rangle
        \end{split}
    \end{align*}
    Therefore, by applying \eqref{ell},
    \begin{align*}
        \begin{split}
            & \re  a( \varepsilon^2+|\na w +\na g|^2 )^\frac{p-2}{2}  \langle (\na w+\na g),\na w  \rangle \\
            &\quad \ge \nu ( \varepsilon^2+|\na w +\na g|^2 )^\frac{p-2}{2} |\na w+\na g|^2  -L (\varepsilon^2+|\na w +\na g|^2 )^\frac{p-1}{2}  |\na g |.
        \end{split}
    \end{align*}
In order to estimate the last term we apply Young's inequality and note that for $\varepsilon, \delta \in (0,1]$:
     \begin{align*}
        \begin{split}
             L (\varepsilon^2+|\na w +\na g|^2 )^\frac{p-1}{2}  |\na g | &\leq  \delta( \varepsilon^2+|\na w +\na g|^2 )^\frac{p}{2}  + c_\delta |\na g |^p 
             \\
            &\leq  c_p \delta |\na w +\na g|^p + c_p + C_\delta |\na g |^p.
        \end{split}
    \end{align*}
    Combining these estimates and integrating over $\Omega$ we have
 \begin{align}\label{est-integrated}
  &\re\int_{\Om} a( \varepsilon^2+|\na w +\na g|^2 )^\frac{p-2}{2}  \langle (\na w+\na g),\na w  \rangle \,dx 
  \\
 \notag & \quad \geq \nu \int_{\Om} ( \varepsilon^2+|\na w +\na g|^2 )^\frac{p-2}{2} |\na w+\na g|^2\, dx - c_p\delta \int_{\Om} |\na w +\na g|^p \, dx
  \\
  \notag &\quad \quad - C_{p,\delta} \int_\Omega |\nabla g|^p +1\, dx.
 \end{align}
    If $p\ge2$, we can simply estimate the first integral on the second row downwards by removing $\varepsilon$. Then, for sufficiently small $\delta$ we have
    \begin{align*}
        \begin{split}
            & \re\int_{\Om} a( \varepsilon^2+|\na w +\na g|^2 )^\frac{p-2}{2}  \langle (\na w+\na g),\na w  \rangle \,dx\\
            &\ge \frac{\nu}{2}\int_{\Om}  |\na w+\na g|^p\,dx-c\int_{\Om}  |\na g |^p+1\,dx
        \end{split}
    \end{align*}
    for some $c=c(p,\nu,L)$. If $p<2$, then we use the argument in the proof of Theorem~\ref{thm_ex} to have
    \begin{align*}
        \begin{split}
            &\int_{\Om}|\na w+\na g|^p\,dx\\
            &\le \left(  \int_\Om (\varepsilon^2+|\na w+\na g|^2)^\frac{p-2}{2}|\na w+\na g|^2\,dx  \right)^\frac{p}{2}  \left( 2\int_{\Om} (\varepsilon^p+|\na w+\na g|^p)\,dx  \right)^\frac{2-p}{2}.
        \end{split}
    \end{align*}
   Applying Young's inequality to the above inequality, we have
   \[
   \int_{\Om}|\na w+\na g|^p\,dx\le c \left(\int_\Om (\varepsilon^2+|\na w+\na g|^2)^\frac{p-2}{2}|\na w+\na g|^2  +1\,dx \right)
   \]
  for some $c=c(p)>0$. Combining this with \eqref{est-integrated} and taking $\delta>0$ sufficiently small we have
  \begin{align*}
        \begin{split}
            & \re\int_{\Om} a( \varepsilon^2+|\na w +\na g|^2 )^\frac{p-2}{2}  \langle (\na w+\na g),\na w  \rangle \,dx\\
            &\ge \frac{1}{c}\int_{\Om}  |\na w+\na g|^p\,dx-c\int_{\Om}  |\na g |^p+1\,dx,
        \end{split}
    \end{align*}
for some sufficiently large $c$. Thus, for all $p>1$ and all $w\in W_0^{1,p}(\Om;\C^N)$,
    \begin{align}\label{goblin}
            &\re\int_{\Om} a( \varepsilon^2+|\na w +\na g|^2 )^\frac{p-2}{2}  \langle (\na w+\na g),\na w  \rangle \,dx-\re \int_\Om \langle F, \na w\rangle\,dx
            \\
            \notag &\quad \ge \frac{1}{c}\int_{\Om}  |\na w+\na g|^p\,dx-c\int_{\Om}  |\na g |^p+1\,dx -\int_{\Om} |F||\na w|\,dx
            \\
            \notag &\quad \ge\frac{1}{c}\int_{\Om}  |\na w|^p\,dx-c\int_{\Om}  |\na g |^p+|F|^p+1\,dx,
    \end{align}
    where we used the triangle inequality and Young's inequality to get the last inequality. Consider again a dense subset $\{\varphi_j\}_{j\in \N}$ of $ W_0^{1,p}(\Om;\C^N)$ and define the map $G: \C^n\to \C^n$ by 
    \begin{align*}
     G_j(z) &:= \int_{\Om} a( \varepsilon^2+|\na w(z) +\na g|^2 )^\frac{p-2}{2}  \langle (\na w(z) + \na g),\na \varphi_j  \rangle - \langle F, \na \varphi_j \rangle\,dx,
     \\
     w(z) &:= \sum^n_{k=1}z_k \varphi_k.
    \end{align*}
Then \eqref{goblin} shows that $G$ satisfies the assumptions of Lemma \ref{Brouwer_cor} and thus for each $m\in\mathbb{N}$, there exists $w_m\in W_m$ such that
     \begin{align}\label{eq-approximate}
    -\dv(a(\varepsilon^2+|\na w_m+\na g|^2)^\frac{p-2}{2}(\na w_m+\na g))= -\dv F \quad\text{in}\quad W_m'.
    \end{align}
    Moreover, the uniform boundedness of $\na w_m$ in $L^{p}(\Om)$ also follows from \eqref{goblin} since the first row of the estimate vanishes when we take $w=w_n$, due to \eqref{eq-approximate}. Therefore, there is a $w\in W_0^{1,p}(\Om;\C^{N })$ and a subsequence still labelled $(w_n)$ such that $w_n$ converges strongly to $w$ and  $\na w_m$ converges weakly to $\na w$. 
    
    To show that $u= g + w$ is indeed a solution to the original equation we proceed similarly as in the last part of the proof of Theorem \ref{thm_ex}.  We denote $u_m=w_m+g$ and let $\psi\in W^{1,p}(\Omega;\C^N)$ be arbitrary. Take a sequence $\tilde v_n$ converging strongly in $W^{1,p}_0(\Omega;\C^N)$ to $w$ and let $v_n := g + \tilde v_n$. With these modifications the argument in the proof of Theorem~\ref{thm_ex} works also in the current case. We omit the details.
\end{proof}

\section{Gradient H\"older continuity}\label{sec:holdcont}
In this section, we consider the regularity properties of solutions to \eqref{eq} when the coefficient function and the source term is H\"older continuous. Suppose that 
\begin{align}\label{Holder_data}
    a\in \mathcal{E}(\Om)\cap C^{\alpha/p'}(\Om;\C),\qquad F\in C^{\alpha/p'}(\Om;\C^{N\times n})
\end{align}
for some $\alpha\in(0,1)$. The main theorem of this section is as follows.

\begin{theorem}\label{schauder}
    Suppose \eqref{Holder_data} holds. Then the gradient of solutions to \eqref{eq} is locally H\"older-continuous. Moreover, there exist $r_0$, $c>0$, $\beta\in(0,\alpha)$ depending on $n,N,p,\nu,L,\alpha,[a]_{\alpha/p'},[F]_{\alpha/p'}$ such that
    \[
    \sup_{x,y\in B_{\rho}(x_0)}|\na u(x)-\na u(y)|\le c \left(\frac{\rho}{r}\right)^\beta\left( \fint_{B_{3r}(x_0)} (|\na u|^p+1)\,dx \right)^\frac{1}{p}
    \]
    for any $0<\rho <r < r_0$ with $B_{3r_0}(x_0)\subset \Om$.
\end{theorem}

The proof is based on the Schauder estimate of the real valued PDEs. For this, the gradient H\"older regularity in the case of a vanishing right-hand side and $a\equiv 1$ is required. Suppose therefore that  $v\in W^{1,p}(\Omega; \C^{N})$ solves
\begin{align}\label{eq:model_case2}
-\dv\big( (\varepsilon^2+|\na v|^2)^\frac{p-2}{2}\na v \big)= 0,
\end{align}
in some domain $\Omega$. Decomposing $v$ into real and imaginary parts as
\[
v=v^R+iv^I,
\]
where $v^R,v^I\in W^{1,p}(\Omega;\mathbb{R}^N)$, we define
\[
w := (v^R, v^I)\in W^{1,p}(\Omega;\mathbb{R}^{2N}).
\]
Then $w$ is a real valued weak solution to
\begin{align}\label{eq:real_val-again}
-\dv \big((\varepsilon^2+|\na w|_{\mathbb{R}^{2N\times n}}^2)^\frac{p-2}{2}\na w\big) = 0,
\end{align}
in $\Omega$. 
To see this, note that for any $\varphi\in W_0^{1,p}(\Omega;\mathbb{R}^{2N})$, we can define a test function $\phi\in W_0^{1,p}(\Omega;\C^N)$ by
\[
\phi= (\varphi_1,...,\varphi_{N})+i(\varphi_{N+1},...,\varphi_{2N}).
\]
The weak formulation corresponding \eqref{eq:model_case2} to states that
\begin{align*}
	\begin{split}
		\int_{\Om} (\varepsilon^2+|\na v|^2)^\frac{p-2}{2} \langle\na v, \na \phi\rangle\,dx=0,
		%\\
		%&=\int_{\Om} (\varepsilon^2+|\na v|^2)^\frac{p-2}{2} (\na v^R \cdot\na \phi^R+ \na v^I \cdot\na \phi^I+i (-\na v^R \cdot\na \phi^I+\na v^I \cdot\na \phi^R ) )\,dx
	\end{split}
\end{align*}
and the real part of this equation is equal to
\[
\int_{\Om}  (\varepsilon^2+|\na w|_{\mathbb{R}^{2N\times n}}^2)^\frac{p-2}{2}\na w\cdot \na \varphi\,dx= 0,
\]
which confirms \eqref{eq:real_val-again}.
Therefore, we can apply the regularity results in \cite{uhlenbeck1977regularity} to $w$ and conclude the following.

\begin{theorem}\label{limit_grad_holder}
    Suppose $v\in W^{1,p}(B_{3\varrho}(x_0),\C^N)$ is a solution to
\begin{align}\label{hom}
    -\dv (\varepsilon^2+|\na v|^2)^\frac{p-2}{2}\na v=0\quad\text{in}\quad B_{3\varrho}(x_0).
\end{align}
Then $\na v$ is locally H\"older-continuous. Moreover, there exist $c>0$ and $\gamma\in(0,1)$ depending on $n,N,p$ such that
\[
\sup_{x\in B_{2\varrho}(x_0)}|\na v(x)|\le c\left(\fint_{B_{3\varrho}(x_0)}  |\na v|^p\,dx +1\right)^\frac{1}{p}
\]
and for any $s\in(0,2\varrho)$,
\[
\sup_{x,y\in B_{s}(x_0)}|\na v(x)-\na v(y)|\le c\left( \frac{s}{2\varrho} \right)^\gamma\left( \fint_{B_{3\varrho}(x_0)}|\na v|^p\,dx+1 \right)^\frac{1}{p}.
\]

\end{theorem}

Hereafter, we start to prove lemmas for the Schauder estimate.
Without loss of generality, we may assume $r\in(0,1)$. For each $\varrho$ such that $\bar B_\varrho(x_0)\subset \Omega$, let 
\begin{align*}
 v=v(\varrho)\in u+W_0^{1,p}(B_{\varrho}(x_0);\C^N )
\end{align*}
be the weak solution to \eqref{hom} in $B_{\varrho}(x_0)$. We remark that a unique solution exists due to Corollary \ref{thm_ex_non_bd}. We begin with two standard energy estimates.
\begin{lemma}\label{limit_en}
There exists $c=c(p)$ such that for every ball $\bar B_\varrho(x_0) \subset \Omega$ and $v=v(\varrho)$ we have
\[
\fint_{B_{\varrho}(x_0)}|\na v|^p\,dx\le c\fint_{B_{\varrho}(x_0)} (|\na u|^p+1)\,dx.
\]
\end{lemma}

\begin{proof}
    We take $v-u$ as a test function to \eqref{hom}.
    Then, we get
\begin{align}\label{p_le_2_en}
0 \ge\fint_{B_{\varrho}(x_0)}(\varepsilon^2+|\na v|^2)^\frac{p-2}{2}(|\na v|^2-|\na v||\na u| ) \,dx.
\end{align}
We divide the proof into cases.

\noindent \textit{Case $p\ge2$.}
Recalling that $\varepsilon\leq 1$, we can use \eqref{p_le_2_en} and Young's inequality to obtain
\begin{align*}
\fint_{B_{\varrho}(x_0)}|\na v|^p\,dx &\le \fint_{B_{\varrho}(x_0)} (\varepsilon^2+|\na v|^2)^\frac{p-2}{2}|\na v|^2\,dx
\\
&\le \fint_{B_{\varrho}(x_0)} (\varepsilon^2+|\na v|^2)^\frac{p-2}{2}|\na v||\na u|\,dx
\\
&\leq c_p \fint_{B_{\varrho}(x_0)}|\na v||\na u| + |\na v|^{p-1} |\na u| \,dx
\\
&\leq \tfrac14 \fint_{B_{\varrho}(x_0)} |\nabla v|^{p'} + |\nabla v|^p \,dx + c \fint_{B_{\varrho}(x_0)} |\na u|^p \, dx.
\end{align*}
Finally, noting that $p\geq p'$ in this case, we can apply Young's inequality again on the first term in the integrand, ending up with
\[
\fint_{B_{\varrho}(x_0)}|\na v|^p\,dx\le c\fint_{B_{\varrho}(x_0)} (|\na u|^p+1)\,dx,
\]
for some $c=c(p)$.

\noindent \textit{Case $p<2$.} Since we have
\begin{align*}
        \begin{split}
            \fint_{B_{\varrho}(x_0)} |\na v|^p\,dx
        &\le \left( \fint_{B_{\varrho}(x_0)} (\varepsilon^2+|\na v|^2)^\frac{p-2}{2}|\na v|^2\,dx \right)^\frac{p}{2}\\
        &\qquad\times\left( \fint_{B_{\varrho}(x_0)}(\varepsilon^2+|\na v|^2)^\frac{p}{2}\,dx \right)^\frac{2-p}{2},
        \end{split}
\end{align*}
and since we may assume that 
\[
1\le \fint_{B_{\varrho}(x_0)}|\na v|^p\,dx,
\]
Young's inequality leads to
\begin{align*}
    \fint_{B_{\varrho}(x_0)} |\na v|^p\,dx\le c \fint_{B_{\varrho}(x_0)} (\varepsilon^2+|\na v|^2)^\frac{p-2}{2}|\na v|^2\,dx,
\end{align*}
for some $c=c(p)$. Combining this estimate with \eqref{p_le_2_en} we have 
\[
\fint_{B_{\varrho}(x_0)} |\na v|^p\,dx\le c\fint_{B_{\varrho}(x_0)} (\varepsilon^2+|\na v|^2)^\frac{p-1}{2}|\na u|\,dx.
\]
Finally, another application of Young's inequality confirms that
\[
\fint_{B_{\varrho}(x_0)}|\na v|^p\,dx\le c\fint_{B_{\varrho}(x_0)}(|\na u|^p+1)\,dx.
\]
\end{proof}

\begin{lemma}\label{comparison}
For each $\varrho\in(0,1]$ such that $\bar B_{\varrho}(x_o) \subset \Omega$ and $v=v(\varrho)$, there exists $c=c(p,\nu,[a]_{\alpha/p'},[F]_{\alpha/p'})$ such that
    \[
\fint_{B_{\varrho}(x_0)} |\na u-\na v|^{p}\,dx\le c\varrho^{\alpha_0}\fint_{B_{\varrho}(x_0)} (|\na u|^p+1)\,dx,
\]
where
\[
\alpha_0=\min \left\{\alpha, \alpha(p-1) \right\}.
\]
\end{lemma}

\begin{proof}
    We take $u-v$ as a test function to 
    \begin{align*}
        \begin{split}
            &-\dv \big( a_0 \big[ (\varepsilon^2+|\na u|^2)^\frac{p-2}{2}\na u - (\varepsilon^2+|\na v|^2)^\frac{p-2}{2}\na v  \big] \big)\\
            &=\dv \big((a-a_0 )(\varepsilon^2+|\na u|^2)^\frac{p-2}{2}\na u  \big)-\dv(F-F_0)
        \end{split}
    \end{align*}
    in $B\equiv B_{\varrho}(x_0)$ where $a_0=a(x_0)$ and $F_0=F(x_0)$. Taking the real part of the resulting expression and using Lemma \ref{str}, Lemma \ref{str1} and Lemma \ref{str2} combined with the condition \eqref{ell2} as in the proof of Theorem \ref{uniqueness}, we end up with
   \begin{align*}
       \begin{split}
           &\fint_{B} (\varepsilon^2+|\na u|^2+|\na v|^2)^\frac{p-2}{2}|\na u-\na v|^2\,dx\\
           &\le c \fint_{B}|a_0-a|(\varepsilon^2+|\na u|^2)^\frac{p-2}{2}|\na u||\na u-\na v|\,dx\\
           &\qquad+c\fint_{B}|F-F_0||\na u-\na v|\,dx.
       \end{split}
   \end{align*}
Using the H\"older continuity conditions \eqref{Holder_data} we can estimate the right hand side upwards to obtain
\begin{align}
 \begin{split}
 \label{wlnt}  &\fint_{B} (\varepsilon^2+|\na u|^2+|\na v|^2)^\frac{p-2}{2}|\na u-\na v|^2\,dx 
   \\
   &\leq c \fint_B \varrho^{\alpha/p'} \big[(\varepsilon^2 +|\nabla u|^2)^\frac{p-1}{2} + 1\big]|\nabla u - \nabla v|\,dx.
 \end{split}
\end{align}
We divide the rest of the proof into cases.

\noindent \textit{Case $p\ge2$.} In this case, we observe that
\[
|\na u - \na v|^p = |\na u-\na v|^{p-2} |\na u - \na v|^2 \le 2^\frac{p-2}{2}( |\na u|^2+|\na v|^2  )^\frac{p-2}{2}|\na u - \na v|^2.
\]
Taking the integral average over $B$ and combining the resulting estimate with \eqref{wlnt} we end up with
\begin{align*}
        \fint_{B} |\na u-\na v|^{p}\,dx
        \le c \fint_B \varrho^{\alpha/p'} \big[(1 +|\nabla u|^2)^\frac{p-1}{2} + 1\big]|\nabla u - \nabla v|\,dx.
\end{align*}
To estimate further on the right hand side, we apply Young's inequality to have
\[
\fint_{B} |\na u-\na v|^{p}\,dx\le c\varrho^\alpha\fint_B (|\na u|^p+1)\,dx.
\]
Thus the conclusion holds in this case.

\noindent \textit{Case $p<2$.} In this case $2/p > 1$, and we can use H\"older's inequality to obtain
\begin{align*}
    \begin{split}
        &\fint_B |\na u-\na v|^p\,dx\\
        &= \fint_B (\varepsilon^2+|\na u|^2+|\na v|^2)^\frac{(p-2)p}{4} |\na u-\na v|^p (\varepsilon^2+|\na u|^2+|\na v|^2)^\frac{(2-p)p}{4} \,dx\\
        &\le \left( \fint_{B} (\varepsilon^2+|\na u|^2+|\na v|^2)^\frac{p-2}{2}|\na u-\na v|^2\,dx \right)^\frac{p}{2}\hspace{-1mm}\left( \fint_B (\varepsilon^2+|\na u|^2+|\na v|^2)^\frac{p}{2}\,dx \right)^\frac{2-p}{2}\hspace{-1mm}.
    \end{split}
\end{align*}
Estimating the first factor using \eqref{wlnt} and applying H\"older inequality yields
\begin{align*}
    \begin{split}
        \fint_B |\na u-\na v|^p\,dx
        &\le c\left( \varrho^\alpha\fint_{B} 1+|\na u|^p \,dx \right)^\frac{p-1}{2}\left( \fint_B |\na u-\na v|^p\,dx \right)^\frac{1}{2}\\
        &\qquad\times\left( \fint_B 1 + |\na u|^p + |\na v|^p\,dx \right)^\frac{2-p}{2}.
    \end{split}
\end{align*}
Applying Lemma~\ref{limit_en} to the last factor and solving for the $L^p$-norm of $\na u - \na v$ completes the proof. In the last step we also use that $\varrho \leq 1$, which guarantees that
\begin{align*}
 \max\{\varrho^\alpha, \varrho^{(p-1)\alpha} \} \leq \varrho^{\alpha_0}.
\end{align*}
\end{proof}

\begin{lemma}\label{decay_est}
    Suppose $\kappa\in(0,n)$.
    There exists $r_0$ and $c$ depending on $n,N,p,\nu,$ $[a]_{\alpha/p'}, [F]_{\alpha/p'},\kappa$ such that if $\bar B(x_0,r_0)\subset \Omega$ then
    for any $0<\varrho<r<r_0$,
    \[
    \int_{B_{\varrho }(x_0)} (|\na u|^p+1)\,dx\le c\left(\frac{\varrho}{r}\right)^{\kappa}\int_{B_{r}(x_0)}(|\na u|^p+1)\,dx.
    \]
\end{lemma}

\begin{proof}
For simplicity, we denote
\begin{align*}
 B:= B_{r}(x_0), \quad  \tau:=\tfrac{\varrho}{r} \in (0,1), \quad  \tau B := B_{\tau r}(x_0) = B_\varrho(x_0). 
\end{align*}
Let $v=v(r)$. Then, by the triangle inequality, we have
    \[
     \fint_{\tau B} (|\na u|^p +1)\,dx  \le 2^{p-1}\fint_{\tau B} (|\na v|^p+1)\,dx  +2^{p-1}\fint_{\tau B} |\na u-\na v|^p\,dx .
    \]
    For the first term, we apply the supremum estimate in Theorem~\ref{limit_grad_holder} if $\tau\in(0,\tfrac{2}{3})$ or replace $\tau B$ in the referenced domain of the integral if $\tau\in[\tfrac{2}{3},1)$. While for the second term, we use Lemma~\ref{comparison} (which is possible if we take $r_0 \leq 1$). Then, we get
    \[
    \fint_{\tau B} (|\na u|^p+1)\,dx  \le c \fint_{ B} (|\na v|^p+1)\,dx  +cr^{\alpha_0}\tau^{-n} \fint_{B} (|\na u|^p+1)\,dx.
    \]
    where $c=c(n,N,p,\nu,[a]_{\alpha/p'},[F]_{\alpha/p'})$.
    To estimate the first term on the right-hand side we again apply the triangle inequality and Lemma~\ref{comparison} to obtain
    \[
    \fint_{\tau B} (|\na u|^p+1)\,dx  \le c \fint_{ B} (|\na u|^p+1)\,dx  +cr^{\alpha_0}\tau^{-n} \fint_{B} (|\na u|^p+1)\,dx.
    \]
    Multiplying with the measure of $\tau B$ this can equivalently be stated as
    \[
    \int_{B_\varrho(x_0)} (|\na u|^p+1)\,dx  \le c \left(\left(\frac{\varrho}{r}\right)^n+r^{\alpha_0}\right)\int_{B_r(x_0)} (|\na u|^p+1)\,dx.
    \]
    This allows us to applying Lemma~\ref{tech}, with the choice $B=0$ and $\varepsilon_0 = r_0^{\alpha_0}$, which completes the proof.
\end{proof}

The following technical lemma can be found in \cite[Lemma 3.4, Chapter 3]{han2011elliptic}.

\begin{lemma}\label{tech}
    Let $\phi(t)$ be a nonnegative and nondecreasing function on $[0,R]$. Suppose that 
    \[
    \phi(\varrho)\le A\left( \left(\frac{\varrho}{r}\right)^\alpha+\varepsilon \right) \phi(r)+Br^\beta
    \]
    for any $0<\varrho\le r\le R$ with $A,B,\alpha,\beta$ nonnegative constants and $\beta<\alpha
    $. Then for any $\gamma\in (\beta,\alpha)$, there exists a constant $\varepsilon_0=\varepsilon_0(A,\alpha,\beta,\gamma)$ such that if $\varepsilon<\varepsilon_0$, we have for all $0<\varrho\le r\le R$
    \[
    \phi(\varrho)\le c\left( \left(\frac{\varrho}{r}\right)^\gamma\phi(r) +B\varrho^\beta \right),
    \]
    where $c$ is a positive constant depending on $A,\alpha,\beta,\gamma$.
\end{lemma}

We now prove the local gradient estimate of $\na u$.
\begin{proof}[Proof of Theorem~\ref{schauder}]
     Let $\varrho\in(0,r]$ and $v=v(3\varrho/2)$. Take $\sigma \in (0, \varrho)$ which will be fixed  later. By the triangle inequality, we get
     \begin{align*}
         \begin{split}
             &\fint_{B_{\sigma}(x_0)}|\na u-(\na u)_{B_{\sigma}(x_0)}|^p\,dx\\
             &\le c\fint_{B_{\sigma}(x_0)}|\na u-(\na v)_{B_{\sigma}(x_0)}|^p\,dx\\
             &\le c\fint_{B_{\sigma}(x_0)}|\na v-(\na v)_{B_{\sigma}(x_0)}|^p\,dx+c\fint_{B_{\sigma}(x_0)}|\na u-\na v|^p\,dx
         \end{split}
     \end{align*}
     for some $c=c(p)$. We estimate the first term by applying Theorem~\ref{limit_grad_holder}, and then Lemma~\ref{limit_en} in order to exchange $\na v$ for $\na u$. In the second integral we apply Lemma~\ref{comparison} and end up with
     \begin{align*}
         \begin{split}
             &\fint_{B_{\sigma}(x_0)}|\na u-(\na u)_{B_{\sigma}(x_0)}|^p\,dx
             \\
             &\le c\left(  \frac{\sigma}{\varrho} \right)^{\gamma p}\fint_{B_{3\varrho/2}(x_0)} (|\na u|^p+1)\,dx+c\left(  \frac{\sigma}{\varrho} \right)^{-n}\varrho^{\alpha_0}\fint_{B_{3\varrho/2}(x_0)}(|\na u|^p+1)\,dx.
         \end{split}
     \end{align*}
     To estimate further, we employ Lemma~\ref{decay_est}. Then, for any $\kappa\in(0,n)$, we get
     \begin{align*}
         \begin{split}
             &\int_{B_{\sigma}(x_0)}|\na u-(\na u)_{B_{\sigma}(x_0)}|^p\,dx
             \\
             &\le c\left ( \left(  \frac{\sigma}{\varrho} \right)^{\gamma p+n}+\varrho^{\alpha_0}   \right)\left( \left(\frac{\varrho}{r}\right)^{\kappa}\int_{B_{3r/2}(x_0)}(|\na u|^p+1)\,dx \right),
         \end{split}
     \end{align*}
     where $c=c(n,N,p,\nu,[a]_{\alpha/p'},[F]_{\alpha/p'},\kappa)$. We let $\varrho^{1+s}=\sigma$ for some sufficiently small $s>0$ which will be determined later. Then the above display becomes
     \begin{align*}
         \begin{split}
             &\int_{B_{\sigma}(x_0)}|\na u-(\na u)_{B_{\sigma}(x_0)}|^p\,dx
             \\
             &\le c\left ( \sigma^{\frac{s(\gamma p+n) + \kappa}{1+s}} + \sigma^{\frac{\alpha_0+\kappa}{1+s}}   \right)   r^{-\kappa}\int_{B_{3r/2}(x_0)}(|\na u|^p+1)\,dx ,
         \end{split}
     \end{align*}
     We take $s:= \tfrac{\alpha_0}{\gamma p+n}$, so that the exponents of $\sigma$ are identical, and thus
     \[
     \int_{B_{\sigma}(x_0)}|\na u-(\na u)_{B_{\sigma}(x_0)}|^p\,dx\le c\sigma^{\frac{s(\gamma p+n)+\kappa}{1+s}}r^{-\kappa}\int_{B_{3r/2}(x_0)}(|\na u|^p+1)\,dx.
     \]
     Choose $\kappa \in (n-s\gamma p, n)$ and $\beta>0$ so that
     \[
     n<\frac{s(\gamma p+n)+\kappa}{1+s}=n+\beta p.
     \]
     Furthermore, since we observe
     \[
     \kappa<\frac{s(\gamma p+n)+\kappa}{1+s},
     \]
     we arrive at
     \[
     \fint_{B_{\sigma}(x_0)}|\na u-(\na u)_{B_{\sigma}(x_0)}|^p\,dx\le c \left(\frac{\sigma}{r}\right)^{\beta p}
     \fint_{B_{3r/2}(x_0)}(|\na u|^p+1)\,dx .
     \]
     Note that this estimate is valid for all $\sigma < r^{1+s}$. 
    We now consider the scaled map 
    \[
    w(y)=\frac{u(x_0+ry)}{r},
    \]
     which is defined on $y\in \bar B_{3r_0/r}(\bar 0) \supset \bar B_3(\bar 0)$. Then  $\na w(y)= \na u(x_0+ry)$, and it follows that $w$ solves an equation similar to that of $u$ with re-scaled coefficient and source term, having the same bound for the H\"older semi-norms as before since $r\leq 1$.  Therefore, by the previous argument, it follows that
     \[
     \fint_{B_{\tau}(\bar 0)}|\na w - (\na w)_{B_{\tau}(\bar 0)}|^p\,dy
     \le c \left(\frac{\tau}{r_0}\right)^{\beta p}
     \fint_{B_{3r_0/2}(\bar 0)}(|\na w|^p+1)\,dy,
     \]
     for all $\tau<r_0^{1+s}$ where $r_0$ is as in Lemma~\ref{decay_est}. Since $r_0$ is a constant, we have
     \begin{align}\label{est:re-scaled}
     \fint_{B_{\tau}(\bar 0)}|\na w - (\na w)_{B_{\tau}}|^p\,dy
     \le c \tau^{\beta p}
     \fint_{B_{3/2}(\bar 0)}(|\na w|^p+1)\,dy,
     \end{align}
     for $\tau<r_0^{1+s}$. On the other hand, if $r_0^{1+s}\le \tau <1$, we apply the triangle inequality to get
     \begin{align*}
     	\fint_{B_{\tau}(\bar 0)}|\na w - (\na w)_{B_{\tau}(\bar 0)}|^p\,dy
     	&\le 2^p\fint_{B_{\tau}(\bar 0)}|\na w|^p\,dy\\
     	&\le 2^p\left( \frac{3}{2r_0^{1+s}}\right)^n\fint_{B_{3/2}(\bar 0)}|\na w|^p\,dy
     	\\
     	&\leq c \tau^{\beta p} \fint_{B_{3/2}(\bar 0)}|\na w|^p\,dy,
     \end{align*}
    where again the factors of $r_0$ have been absorbed into the constant $c$. Combining the estimates for the different ranges of $\tau$ we conclude that \eqref{est:re-scaled} is valid for all $0<\tau<1$. Therefore, scaling back to the original map, we have for any $\sigma\in(0,r)$ the estimate
    \[
    \fint_{B_{\sigma}(x_0)}|\na u-(\na u)_{B_{\sigma}(x_0)}|^p\,dx\le c \left(\frac{\sigma}{r}\right)^{\beta p}
    \fint_{B_{3r/2}(x_0)}(|\na u|^p+1)\,dx .
    \]
     Since $B_{3r_0}(x_0) \subset \Omega$, a similar estimate can be obtained if we replace $x_0$ by any point $x_1 \in B_\sigma(x_0)$ and we can also replace $\sigma$ in the estimate by any $\tilde \sigma \leq \sigma$. The proof is therefore completed by the Campanato characterization and covering argument, see for example  \cite{campanato1963}.
\end{proof}

\section{Application to the nonlinear parametric problem}
In this section, we investigate how the solution of the equation \eqref{eq} depends on the coefficient function \eqref{coeff}. Specifically we consider the case of coefficient functions $a(z)$ parametrized by a single complex variable $z\in \C$ confined to some open subset $U\subset \C$. Thus, for all $z\in U$ we suppose that $a(z)$ belongs to the space of coefficient functions $\mathcal{E}(\Om)$ for some fixed constants $0<\nu<L$.
For fixed $F\in L^{p'}(\Om;\C^{N\times n})$, we can then by the existence and uniqueness results obtained in Section \ref{sec:existence}  define the map $\mathcal{F}:U\to W_0^{1,p}(\Om;\C^N)$, $\mathcal{F}(z)=u(z)$ where $u(z)$ is the unique weak solution in $ W_0^{1,p}(\Om;\C^N)$ to
\[
-\dv ( a(z)(\varepsilon^2+|\na u(z)|^2)^\frac{p-2}{2}\na u(z) )=-\dv F.
\]

The aim of this section is to study the regularity properties of $\mathcal{F}$.
For this purpose we will employ the Schauder estimate obtained in the previous section. 
As we need a global version of the estimate, we restrict ourselves to domains of the form
\begin{align}\label{cube}
    \Om=C_R=\{ x=(x_1,...,x_n)\in \mathbb{R}^n: |x_j|\le R\quad\text{for all}\quad 1\le j\le n \}.
\end{align}
for some $R>0$. We remark that the subsequent arguments and results are valid also for other rectangular domains.

\begin{theorem}\label{global_schauder}
        Suppose \eqref{cube} and \eqref{Holder_data} hold. Then the gradient of solutions to \eqref{eq} is H\"older-continuous in $C_R$. Moreover, there exist $r_0$, $c>0$ and $\beta\in(0,\alpha)$  depending on $n,N,p,\nu,L,\alpha,[a]_{\alpha/p'},[F]_{\alpha/p'}, R$ such that
    \[
    \sup_{x,y\in \cap C_R}|\na u(x)-\na u(y)|\le c |x-y|^\beta\left( \int_{C_R} (|\na u|^p+1)\,dx \right)^\frac{1}{p}.
    \]
\end{theorem}

The proof is analogous to the proof of Theorem~\ref{schauder} as we are considering $C_R$, the reflection argument can be employed to the limiting PDE so that the regularity properties and corresponding estimates hold up to the referenced boundary. We refer to \cite{Ya1989,MR1230384} and omit the details.

We now state the main theorem of this section. We remark that by Theorem \ref{global_schauder}, we have $\mathcal{F}:U\mapsto W_0^{1,\infty}(C_R,\C^N)\subset W_0^{1,2}(C_R,\C^N)$.

\begin{theorem}\label{diff_thm}
    Let $\varepsilon\in(0,1)$, $\alpha\in(0,1)$ and let $U$ be an open subset in $\C$. Let $\Omega$ be a rectangular domain in $\R^n$. Let $F\in C^{\alpha/p'}(\Omega;\C^{N\times n})$. 
    Suppose $a$ is a map from $U$ to $C^{\alpha/p'}(\Omega;\C)$ with $a(z)\in\mathcal{E}(\Omega)$ for all $z\in U$ with fixed constants $0<\nu<L$ and 
    \[
    \sup_{z\in U}[a(z)]_{C^{\alpha/p'}(\Omega;\C)}<\infty.
    \]
    Denote by $u(z)$ the unique weak solution in $ W_0^{1,p}(\Om;\C^N)$ to problem \eqref{eq} with the coefficient function $a(z)$, $z \in \C$. Define the map $\mathcal{F}:U\to W_0^{1,\infty}(\Omega;\C^N)$,
    
    \noindent $\mathcal{F}(z) = u(z)$.
    Furthermore, suppose $a'(z)\in L^\infty(\Omega;\C)$, that is,
    \[
    \lim_{h\to 0}|h|^{-1}\| a(z+h)-a(z)-h a'(z) \|_{L^\infty(\Omega;\C)}=0,
    \]
    Suppose also that the complex coefficient satisfies the condition
    \begin{align}\label{a_extra-cond}
     s a^R(z) > \frac{|p-2|}{p}|a(z)|,
    \end{align}
    for some $s \in (0,1]$. Then for every $\theta \in \mathbb{S}^1$, there is $w_\theta \in W^{1,2}_0(\Omega;\C^N)$ such that
\begin{align}\label{weak_convg_on_line_seggg}
\frac{u(z+t\theta) - u(z)}{t\theta} \xrightarrow[t\to 0]{} w_\theta \quad \text{weakly in } W^{1,2}_0(\Omega; \C^N).
\end{align}If furthermore \eqref{a_extra-cond} holds for some $s\in (0,1)$ then $w_\theta$ is in fact the derivative of $\mathcal{F}$ as a map into $W^{1,2}_0$ along the line segment $z + t\theta, t \in \R$. That is,
\begin{align}\label{diffability_along_line_seg}
 \lim_{t\to 0} \frac{\norm{u(z + t\theta) - u(z) - t\theta w_\theta}_{W^{1,2}_0(\Omega;\C^N)}}{t} = 0.
\end{align}

\end{theorem}
\begin{remark}
 Given the nonnegativity of $a^R$, the extra condition \eqref{a_extra-cond} can be re-written as
 \begin{align*}
  |a^I(z)| < a^R(z)\Big(\big(\tfrac{s p}{p-2}\big)^2 - 1 \Big)^\frac12,
 \end{align*}
which essentially says that the imaginary part should not be too large compared to the real part.
\end{remark}

\begin{proof}[Proof of Theorem \ref{diff_thm}]
We denote $\mathcal{F}(z)=u(z)$ and $u(z)(x)=u(z,x)$ for $z\in U$ and $x\in \Omega$.
    By Theorem~\ref{global_schauder}, we have that
    \begin{align*}
        \begin{split}
            \sup_{x\in \Omega}|\na u(z,x)|
            &\le |(\na u(z))_{\Omega}|+c\left(\int_{\Omega} (|\na u(z,x)|^p+1)\,dx\right)^\frac{1}{p}\\
            &\le c\left(\int_{\Omega} (|\na u(z,x)|^p+1)\,dx\right)^\frac{1}{p}
        \end{split}
    \end{align*}
   for some constant $c>0$, not depending on $z$. Moreover, since the source term $F$ is fixed, it follows that there exists a uniform bound $c>0$ such that
   \begin{align}\label{uniform_sup}
       \sup_{z\in U}\sup_{x\in \Omega}|\na u(z,x)|\le c.
   \end{align}
    We consider
    \begin{align*}
        \begin{split}
            & -\dv a(z+h)(\varepsilon^2+|\na u(z+h)|^2)^\frac{p-2}{2}\na u(z+h)  = -\dv F, \\
            & -\dv a(z)(\varepsilon^2+|\na u(z)|^2)^\frac{p-2}{2}\na u(z)  = -\dv F.
        \end{split}
    \end{align*}
    By taking $u(z) - u(z+h)$ as a test function to
    \begin{align*}
        \begin{split}
           & -\dv \Big( a(z) \big[ (\varepsilon^2+|\na u(z)|^2)^\frac{p-2}{2}\na u(z)- (\varepsilon^2+|\na u(z+h)|^2)^\frac{p-2}{2}\na u(z+h) \big]\Big)\\
           &= -\dv\Big( (a(z+h)-a(z))(\varepsilon^2+|\na u(z+h)|^2)^\frac{p-2}{2}\na u(z+h)\Big),
        \end{split}
    \end{align*}
    we have from Lemma~\ref{str3} and \eqref{uniform_sup} that
    \begin{align*}
       \int_{\Omega}|\na u(z+h)-\na u(z)|^2\,dx\le c\int_{\Omega} |a(z+h)-a(z)||\na u(z)-\na u(z+h)|\,dx,
    \end{align*}
    where we used the fact that $(\varepsilon^2+|\na u(z)|^2+|\na u(z+h)|^2)^\frac{p-2}{2}$ is bounded below and above uniformly in $x$ and $z$ since \eqref{uniform_sup} and $\varepsilon\ne 0$ hold. Therefore, the Young inequality leads to
    \begin{align}\label{diff}
        \int_{\Omega}|\na u(z+h)-\na u(z)|^2\,dx\le c|h|^2,
    \end{align}
    for sufficiently small $|h|>0$.
    By \eqref{uniform_sup}, any sequence $(h_j)$ in $\C$ converging to zero has a subsequence $(h_{j_k})$ converging weakly in $W_0^{1,p}(\Omega;\C^N)$ to some $\eta\in W_0^{1,p}(\Omega;\C^N)$.
    On the other hand, as the functions $\na u(z+h_{j_k})$ are uniformly bounded and $\beta$-H\"older continuous, we conclude by the Arzelà–Ascoli that yet another subsequence still labelled $(h_{j_k})$ converges uniformly to $\na \eta$ and that $\na \eta\in C^\beta(\Omega; \C^{N\times n})$. Thus,
    \[
    \|\na u(z+h_{j_k})-\na \eta\|_{L^\infty(\Omega; \C^{N\times n})}\to 0\quad \text{as}\quad k\to \infty.
    \]
   Moreover, this gives
   \[
   -\dv (a(z)(\varepsilon^2+|\na \eta|^2)^\frac{p-2}{2}\na \eta)=-\dv F.
   \]
    By the uniqueness, we have $\eta\equiv u$. Since any sequence $h_j$ converging to zero has a subsequence such that the gradients converge uniformly to $\na u(z)$, we conclude that
    \begin{align}\label{diff_infty}
        \|\na u(z+h) - \na u(z)\|_{L^\infty(\Omega; \C^{N\times n})}\to 0\quad \text{as}\quad h\to 0.
    \end{align}
We now show that \eqref{weak_convg_on_line_seggg} holds provided that the coefficient functions satisfy \eqref{a_extra-cond}. The first step is to notice that by \eqref{diff} and the Poincar\'e inequality, the difference quotients
\begin{align*}
 \frac{u(z+h) - u(z)}{h}, \quad h \in U
\end{align*}
form a bounded set in $W^{1,2}_0(\Omega; \C^N)$. We combine the weak formulations satisfied by $u(z)$ and $u(z+h)$, and add and remove a term involving the derivative $a'(z)$ to conclude
\begin{align}\label{eq_z_and_z+h_combined}
 \notag 0 &= \int_\Omega [a(z+h) - a(z) - h a'(z)]\big(\varepsilon^2 + |\nabla u(z+h)|^2\big)^{\frac{p-2}{2}} \langle \nabla u(z+h),\nabla \varphi\rangle dx
 \\
  &\quad +\int_\Omega  h a'(z)\big(\varepsilon^2 + |\nabla u(z+h)|^2\big)^{\frac{p-2}{2}}\langle \nabla u(z+h), \nabla \varphi\rangle dx
 \\
 \notag & + \int_\Omega a(z)\big\langle \big(\varepsilon^2 + |\nabla u(z+h)|^2\big)^{\frac{p-2}{2}} \nabla u(z+h) - \big(\varepsilon^2 + |\nabla u(z)|^2\big)^{\frac{p-2}{2}}\nabla u(z), \nabla \varphi\big\rangle d x.
\end{align}
The last integral can further be split as follows:
\begin{align}\label{expansion1}
 \notag &\int_\Omega a(z)\big\langle \big(\varepsilon^2 + |\nabla u(z+h)|^2\big)^{\frac{p-2}{2}} \nabla u(z+h) - \big(\varepsilon^2 + |\nabla u(z)|^2\big)^{\frac{p-2}{2}}\nabla u(z), \nabla \varphi\big\rangle d x 
 \\
 \notag &= \int_\Omega a(z)\big\langle \Big[\big(\varepsilon^2 + |\nabla u(z+h)|^2\big)^{\frac{p-2}{2}} - \big(\varepsilon^2 + |\nabla u(z)|^2\big)^{\frac{p-2}{2}}\Big] \nabla u(z + h), \nabla \varphi\big\rangle d x
 \\
 &\quad + \int_\Omega a(z)\big(\varepsilon^2 + |\nabla u(z)|^2\big)^{\frac{p-2}{2}} \langle \nabla u(z+h) - \nabla u(z), \nabla \varphi \rangle d x.
\end{align}
Using the shorthand notation
\begin{align*}
 U(s,z,h) = s \nabla u (z+h) + (1-s)\nabla u(z),
\end{align*}
the expression inside the square brackets can furthermore be written as
\begin{align}\label{expansion2}
 &\big(\varepsilon^2 + |\nabla u(z+h)|^2\big)^{\frac{p-2}{2}} - \big(\varepsilon^2 + |\nabla u(z)|^2\big)^{\frac{p-2}{2}}
 \\
 \notag &= \int^1_0 \frac{d}{ds} \big(\varepsilon^2 + |U(s,z,h)|^2\big)^{\frac{p-2}{2}}d s
 \\
\notag &=\big(\tfrac{p-2}{2}\big)\int^1_0 \big(\varepsilon^2 + |U(s,z,h)|^2\big)^{\frac{p-4}{2}} \langle U(s,z,h), \nabla u(z+h) - \nabla u(z)\rangle
 \\
\notag &+ \big(\tfrac{p-2}{2}\big)\int^1_0 \big(\varepsilon^2 + |U(s,z,h)|^2\big)^{\frac{p-4}{2}}\langle \nabla u(z+h) - \nabla u(z), U(s,z,h)\rangle d s.
\end{align}
Combining \eqref{eq_z_and_z+h_combined}, \eqref{expansion1} and \eqref{expansion2}, dividing by $h$ and utilizing the sesquilinearity of the inner product we end up with
\begin{align}\label{long_expr}
\notag 0 & = \int_\Omega \Big[  \frac{a(z+h) - a(z) - h a'(z)}{h}\Big]\big(\varepsilon^2 + |\nabla u(z+h)|^2\big)^{\frac{p-2}{2}} \langle \nabla u(z+h),\nabla \varphi\rangle dx
 \\
 & \quad + \int_\Omega  a'(z)\big(\varepsilon^2 + |\nabla u(z+h)|^2\big)^{\frac{p-2}{2}}\langle \nabla u(z+h), \nabla \varphi\rangle dx
\\
\notag & \quad + \int_\Omega a(z)\big(\varepsilon^2 + |\nabla u(z)|^2\big)^{\frac{p-2}{2}} \Big\langle  \frac{\nabla u(z+h) - \nabla u(z)}{h}, \nabla \varphi \Big\rangle d x
\\
\notag & \quad + \big(\tfrac{p-2}{2}\big) \int_\Omega  \Big[\int^1_0 \big(\varepsilon^2 + |U(s,z,h)|^2\big)^{\frac{p-4}{2}}\Big \langle \frac{\nabla u(z+h) - \nabla u(z)}{h}, U(s,z,h)\Big\rangle d s\Big]
\\
\notag &\hspace{82mm} \times a(z) \langle \nabla u(z + h), \nabla \varphi \rangle d x
\\
\notag & \quad + \frac{\bar h}{h}\big(\tfrac{p-2}{2}\big) \int_\Omega \Big[ \int^1_0 \big(\varepsilon^2 + |U(s,z,h)|^2\big)^{\frac{p-4}{2}}\Big \langle U(s,z,h), \frac{\nabla u(z+h) - \nabla u(z)}{h}\Big\rangle d s \Big]
\\
\notag & \hspace{82mm} \times a(z)\langle \nabla u(z + h), \nabla \varphi \rangle d x.
\end{align}
Due to \eqref{diff_infty} we have that 
\begin{align*}
 \norm{U(s,z,h) - \na u(z) }_{L^\infty(\Omega;\C^{N\times n})} \leq \norm{\nabla u(z+h) - \nabla u(z)}_{L^\infty(\Omega;\C^{N\times n})} \to 0.
\end{align*}
Since furthermore, the difference quotients appearing in the last two terms of \eqref{long_expr} remain bounded in $L^2(\Omega;\C^{N\times n})$, we see by adding and removing terms involving $\nabla u(z)$ that we can write
\begin{align}\label{long_expr2}
\notag &0  = \int_\Omega \Big[  \frac{a(z+h) - a(z) - h a'(z)}{h}\Big]\big(\varepsilon^2 + |\nabla u(z+h)|^2\big)^{\frac{p-2}{2}} \langle \nabla u(z+h),\nabla \varphi\rangle dx
 \\
 & \quad + \int_\Omega  a'(z)\big(\varepsilon^2 + |\nabla u(z+h)|^2\big)^{\frac{p-2}{2}}\langle \nabla u(z+h), \nabla \varphi\rangle dx
\\
\notag & \quad + \int_\Omega a(z)\big(\varepsilon^2 + |\nabla u(z)|^2\big)^{\frac{p-2}{2}} \Big\langle  \frac{\nabla u(z+h) - \nabla u(z)}{h}, \nabla \varphi \Big\rangle d x
\\
\notag & \quad + \big(\tfrac{p-2}{2}\big) \int_\Omega a(z)  \big(\varepsilon^2 + |\nabla u(z)|^2\big)^{\frac{p-4}{2}}\Big \langle \frac{\nabla u(z+h) - \nabla u(z)}{h}, \nabla u(z) \Big\rangle  \langle \nabla u(z), \nabla \varphi\rangle d x
\\
\notag &  + \frac{\bar h}{h}\big(\tfrac{p-2}{2}\big) \int_\Omega a(z)  \big(\varepsilon^2 + |\nabla u(z)|^2\big)^{\frac{p-4}{2}}\Big \langle \nabla u(z), \frac{\nabla u(z+h) - \nabla u(z)}{h}\Big\rangle  \langle \nabla u(z), \nabla \varphi \rangle d x
\\
\notag & + R(h),
\end{align}
where $R(h)$ collects remainder terms which converge to zero as $h\to 0$. Fix $\theta \in \mathbb{S}^1$ and let $(t_j)$ be a sequence converging to zero. Then by weak compactness, there is an element $w_\theta \in W^{1,2}_0(\Omega;\C^N)$ and a subsequence, still labelled $(t_j)$, such that
\begin{align*}
 \frac{u(z+t_j\theta) - u(z)}{t_j\theta} \xrightarrow[j\to \infty]{} w_\theta, \quad \text{weakly in } W^{1,2}_0(\Omega;\C^N).
\end{align*}
In order to conclude \eqref{weak_convg_on_line_seggg} we show that $w_\theta$ is independent of the sequence $(t_j)$. To do so, denote $\tilde \theta := \bar \theta/\theta$ and substitute $h=t_j\theta$ in \eqref{long_expr2}. Due to the weak limit of the difference quotient we have by taking $j\to \infty$ that 
\begin{align}\label{eq:w_theta}
& \int_\Omega a(z)\big(\varepsilon^2 + |\nabla u(z)|^2\big)^{\frac{p-2}{2}} \langle \nabla w_\theta, \nabla \varphi \rangle d x
 \\
\notag &\quad + \big(\tfrac{p-2}{2}\big) \int_\Omega a(z)  \big(\varepsilon^2 + |\nabla u(z)|^2\big)^{\frac{p-4}{2}}\langle \nabla w_\theta, \nabla u(z) \rangle  \langle \nabla u(z), \nabla \varphi\rangle d x
\\
\notag &\quad + \tilde \theta \big(\tfrac{p-2}{2}\big) \int_\Omega a(z)  \big(\varepsilon^2 + |\nabla u(z)|^2\big)^{\frac{p-4}{2}} \langle \nabla u(z), \nabla w_\theta \rangle  \langle \nabla u(z), \nabla \varphi \rangle d x
 \\
\notag &= -\int_\Omega  a'(z)\big(\varepsilon^2 + |\nabla u(z)|^2\big)^{\frac{p-2}{2}}\langle \nabla u(z), \nabla \varphi\rangle dx,
\end{align}
for all test functions. To see that this condition uniquely determines $w_\theta$, suppose that there are two solutions $w^1$ and $w^2$ satisfying this condition. Then the difference $v:= w^1 - w^2$ satisfies
\begin{align*}
 & \int_\Omega a(z)\big(\varepsilon^2 + |\nabla u(z)|^2\big)^{\frac{p-2}{2}} \langle \nabla v, \nabla \varphi \rangle d x
 \\
 &\quad + \big(\tfrac{p-2}{2}\big) \int_\Omega a(z)  \big(\varepsilon^2 + |\nabla u(z)|^2\big)^{\frac{p-4}{2}}\langle \nabla v, \nabla u(z) \rangle  \langle \nabla u(z), \nabla \varphi\rangle d x
\\
 &\quad + \tilde \theta \big(\tfrac{p-2}{2}\big) \int_\Omega a(z)  \big(\varepsilon^2 + |\nabla u(z)|^2\big)^{\frac{p-4}{2}} \langle \nabla u(z), \nabla v \rangle  \langle \nabla u(z), \nabla \varphi \rangle d x = 0.
\end{align*}
Choose $\varphi = v$ and take the real part to obtain
\begin{align*}
 & \int_\Omega a^R(z)\big(\varepsilon^2 + |\nabla u(z)|^2\big)^{\frac{p-2}{2}} |\nabla v|^2 d x
 \\
 &\quad + \big(\tfrac{p-2}{2}\big) \int_\Omega a^R(z)  \big(\varepsilon^2 + |\nabla u(z)|^2\big)^{\frac{p-4}{2}}|\langle \nabla v, \nabla u(z) \rangle|^2 d x
\\
 &\quad +  \big(\tfrac{p-2}{2}\big) \int_\Omega \big(\varepsilon^2 + |\nabla u(z)|^2\big)^{\frac{p-4}{2}} \re\big(\tilde \theta a(z) \langle \nabla u(z), \nabla v \rangle^2 \big) dx = 0.
\end{align*}
Estimating the real part by the modulus we have
\begin{align*}
 \int_\Omega \big(\varepsilon^2 + |\nabla u(z)|^2\big)^{\frac{p-2}{2}} \Big[a^R(z)|\nabla v|^2 + \big(\tfrac{(p-2)}{2} a^R(z) - \tfrac{|p-2|}{2}|a(z)|\big) \frac{|\langle \nabla u(z),\nabla v\rangle|^2}{\varepsilon^2 + |\nabla u(z)|^2}\Big] d x 
 \\
 \leq 0.
\end{align*}
By the Cauchy-Schwarz inequality we have
\begin{align*}
 \frac{|\langle \nabla u(z),\nabla v\rangle|^2}{\varepsilon^2 + |\nabla u(z)|^2} \leq \frac{|\nabla u(z)|^2 |\nabla v|^2}{\varepsilon^2 + |\nabla u(z)|^2} \leq |\nabla v|^2,
\end{align*}
and noting that the factor in front of the corresponding expression is nonpositive, we may use this estimate ending up with
\begin{align*}
 \int_\Omega \big(\varepsilon^2 + |\nabla u(z)|^2\big)^{\frac{p-2}{2}} \big(p\,a^R(z) - |p-2||a(z)|\big)|\nabla v|^2 d x \leq 0.
\end{align*}
which, combined with \eqref{a_extra-cond} implies that $v\equiv 0$. This confirms the uniqueness of $w_\theta$ and hence that \eqref{weak_convg_on_line_seggg} holds. Now suppose that \eqref{a_extra-cond} holds with $s\in (0,1)$. It remains to prove \eqref{diffability_along_line_seg}. 
Taking $h=t\theta$ for a fixed $\theta \in \mathbb{S}^1$, multiplying \eqref{long_expr2} and \eqref{eq:w_theta} by $h$ and adding the resulting equations we end up with
\begin{align*}
 & \int_\Omega a(z)\big(\varepsilon^2 + |\nabla u(z)|^2\big)^{\frac{p-2}{2}} \langle \nabla u(z+h) - \nabla u(z) - h \nabla w_\theta, \nabla \varphi \rangle d x
 \\
 &+ c_p \hspace{-1mm}\int_\Omega a(z)  \big(\varepsilon^2 + |\nabla u(z)|^2\big)^{\frac{p-4}{2}}\langle \nabla u(z+h) - \nabla u(z) - h \nabla w_\theta, \nabla u(z) \rangle  \langle \nabla u(z), \nabla \varphi\rangle d x
 \\
 &+ c_p \hspace{-1mm}\int_\Omega a(z)  \big(\varepsilon^2 + |\nabla u(z)|^2\big)^{\frac{p-4}{2}}\langle \nabla u(z), \nabla u(z+h) - \nabla u(z) - h \nabla w_\theta \rangle  \langle \nabla u(z), \nabla \varphi\rangle d x
 \\
 &+ h \hspace{-1mm}\int_\Omega a'(z) \big\langle \big(\varepsilon^2 + |\nabla u(z+h)|^2\big)^{\frac{p-2}{2}} \nabla u(z+h) - \big(\varepsilon^2 + |\nabla u(z)|^2\big)^{\frac{p-2}{2}}\nabla u(z), \nabla \varphi\big\rangle d x
 \\
 &+ \int_\Omega [a(z+h) - a(z) - h a'(z)]\big(\varepsilon^2 + |\nabla u(z+h)|^2\big)^{\frac{p-2}{2}} \langle \nabla u(z+h),\nabla \varphi\rangle dx
 \\
 &+ h R(h) = 0,
\end{align*}
where we denote $c_p := \tfrac{p-2}{2}$. By examining the remainder term $R(h)$ we see that last three terms can be combined into a remainder of the form
\begin{align*}
 h \tilde R(h), \quad |\tilde R(h)| \leq \xi(h) \norm{\nabla \varphi}_{L^2(\Omega;\C^{N\times n})}, \quad \lim_{h\to 0} \xi(h) = 0.
\end{align*}
Choosing the test function $\varphi = u(z+h) - u(z) - h w_\theta$ we have
\begin{align*}
 &\int_\Omega a(z)\big(\varepsilon^2 + |\nabla u(z)|^2\big)^{\frac{p-2}{2}}|\nabla \varphi |^2 d x
 \\
 & \quad + \big(\tfrac{p-2}{2}\big) \int_\Omega a(z)  \big(\varepsilon^2 + |\nabla u(z)|^2\big)^{\frac{p-4}{2}} |\langle \nabla u(z), \nabla
 \varphi\rangle|^2 d x
 \\
 & \quad + \big(\tfrac{p-2}{2}\big) \int_\Omega a(z)  \big(\varepsilon^2 + |\nabla u(z)|^2\big)^{\frac{p-4}{2}} (\langle \nabla u(z), \nabla
 \varphi\rangle)^2 d x = h\tilde R(h).
\end{align*}
Taking the real part and estimating as previously we obtain 
\begin{align*}
 \int_\Omega \big(\varepsilon^2 + |\nabla u(z)|^2\big)^{\frac{p-2}{2}} \big(p\,a^R(z) - |p-2||a(z)|\big)|\nabla \varphi|^2d x \leq |h| \xi(h) \norm{\nabla \varphi}_{L^2(\Omega;\C^{N\times n})},
\end{align*}
for some function $\xi(h)$ converging to $0$ as $h\to 0$. Using \eqref{a_extra-cond}, $\varepsilon > 0$ and the boundedness of $\nabla u(z)$, we have
\begin{align*}
 c \norm{\nabla \varphi}_{L^2(\Omega)}^2 \leq \int_\Omega \big(\varepsilon^2 + |\nabla u(z)|^2\big)^{\frac{p-2}{2}} p(1-s)a^R(z)|\nabla \varphi|^2d x 
 \leq |h| \xi(h) \norm{\nabla \varphi}_{L^2(\Omega)}.
\end{align*}
Recalling the definition of $\varphi$ and the fact that $h=t\theta$, we have verified \eqref{diffability_along_line_seg}.
\end{proof}

\bibliography{complex}

\begin{thebibliography}{10}

\bibitem{acerbi1989regularity}
E.~Acerbi and N.~Fusco.
\newblock Regularity for minimizers of non-quadratic functionals: the case $1<
  p< 2$.
\newblock {\em Journal of mathematical analysis and applications},
  140(1):115--135, 1989.

\bibitem{acerbi2002regularity}
E.~Acerbi and G.~Mingione.
\newblock Regularity results for stationary electro-rheological fluids.
\newblock {\em Arch. Ration. Mech. Anal.}, 164:213--259, 2002.

\bibitem{bogelein2012holder}
V.~B{\"o}gelein and F.~Duzaar.
\newblock H{\"o}lder estimates for parabolic p (x, t)-laplacian systems.
\newblock {\em Math. Ann.}, 354(3), 2012.

\bibitem{bögelein2023holdercontinuitygradientsolutions}
V.~Bögelein, F.~Duzaar, U.~Gianazza, N.~Liao, and C.~Scheven.
\newblock H\"older continuity of the gradient of solutions to doubly non-linear
  parabolic equations, 2023.

\bibitem{campanato1963}
S.~Campanato.
\newblock Propriet{\`a} di h{\"o}lderianit{\`a} di alcune classi di funzioni.
\newblock {\em Annali della Scuola Normale Superiore di Pisa - Classe di
  Scienze}, 17(1--2):175--188, 1963.

\bibitem{cohen2015approximation}
A.~Cohen and R.~DeVore.
\newblock Approximation of high-dimensional parametric pdes.
\newblock {\em Acta Numer.}, 24:1--159, 2015.

\bibitem{cohen2011analytic}
A.~Cohen, R.~Devore, and C.~Schwab.
\newblock Analytic regularity and polynomial approximation of parametric and
  stochastic elliptic pde's.
\newblock {\em Anal. Appl.}, 9(01):11--47, 2011.

\bibitem{coscia1999holder}
A.~Coscia and G.~Mingione.
\newblock H{\"o}lder continuity of the gradient of p (x)-harmonic mappings.
\newblock {\em Comptes Rendus de l'Acad{\'e}mie des Sciences-Series
  I-Mathematics}, 328(4):363--368, 1999.

\bibitem{de2023nonuniformly}
C.~De~Filippis and G.~Mingione.
\newblock Nonuniformly elliptic schauder theory.
\newblock {\em Invent. Math.}, 234(3):1109--1196, 2023.

\bibitem{MR1230384}
E.~DiBenedetto.
\newblock {\em Degenerate parabolic equations}.
\newblock Universitext. Springer-Verlag, New York, 1993.

\bibitem{MR1962933}
E.~Giusti.
\newblock {\em Direct methods in the calculus of variations}.
\newblock World Scientific Publishing Co., Inc., River Edge, NJ, 2003.

\bibitem{han2011elliptic}
Q.~Han and F.~Lin.
\newblock {\em Elliptic partial differential equations}, volume~1.
\newblock American Mathematical Soc., 2011.

\bibitem{kuusi2012new}
T.~Kuusi and G.~Mingione.
\newblock New perturbation methods for nonlinear parabolic problems.
\newblock {\em J. Math. Pures Appl.}, 98(4):390--427, 2012.

\bibitem{kuusi2013gradient}
T.~Kuusi and G.~Mingione.
\newblock Gradient regularity for nonlinear parabolic equations.
\newblock {\em Ann. Sc. norm. super. Pisa - Cl. sci.}, 12(4):755--822, 2013.

\bibitem{LanthMishKarn}
S.~Lanthaler, S.~Mishra, and G.~E. Karniadakis.
\newblock Error estimates for deeponets: a deep learning framework in infinite
  dimensions.
\newblock {\em Transactions of Mathematics and Its Applications}, 6(1):tnac001,
  2022.

\bibitem{lions1969quelques}
J.L. Lions.
\newblock {\em Quelques m{\'e}thodes de r{\'e}solution des probl{\`e}mes aux
  limites non lin{\'e}aires}.
\newblock Collection {\'e}tudes math{\'e}matiques. Dunod, 1969.

\bibitem{misawa2002local}
M.~Misawa.
\newblock Local h{\"o}lder regularity of gradients for evolutional p-laplacian
  systems.
\newblock {\em Ann. Mat. Pura Appl.}, 181:389--405, 2002.

\bibitem{MR1422252}
R.~E. Showalter.
\newblock {\em Monotone operators in {B}anach space and nonlinear partial
  differential equations}, volume~49 of {\em Mathematical Surveys and
  Monographs}.
\newblock American Mathematical Society, Providence, RI, 1997.

\bibitem{uhlenbeck1977regularity}
K.~Uhlenbeck.
\newblock Regularity for a class of non-linear elliptic systems.
\newblock {\em Acta Math.}, 138:219--240, 1977.

\bibitem{Ya1989}
C.~Ya-Zhe and E.~DiBenedetto.
\newblock Boundary estimates for solutions of nonlinear degenerate parabolic
  systems.
\newblock {\em J. Reine Angew. Math.}, 395:102--131, 1989.

\end{thebibliography}
\bibliographystyle{plain}

\end{document}